\documentclass[a4paper,12pt]{article}
\usepackage{amsmath}
\usepackage{amsthm}
\usepackage{amsfonts}
\usepackage{amssymb}
\usepackage{setspace}

\newtheorem{lemma}{Lemma}
\newtheorem{proposition}{Proposition}
\newtheorem{corollary}{Corollary}
\newtheorem{theorem}{Theorem}
\newtheorem{definition}{Definition}

\newtheorem{example}{Example}
\newtheorem{remark}{Remark}

\newtheorem*{ack}{Acknowledgment}

\newcommand{\new}{\newcommand}

\new{\beeq}[2]{\begin{equation}\label{#1}{#2}\end{equation}}
\new{\eqn}[1]{~(\ref{#1})}

\new{\de}{{\rm d}}

\new{\unichiusZ}{\no{\cdot}_Z{\rm -cl}\,}

\new{\frecc}{\longrightarrow}
\new{\ev}{{\rm ev}}
\new{\Spanno}[1]{{\rm span\,}\left\{ {#1} \right\}}
\new{\Spannochiuso}[1]{\overline{\rm span}\,\left\{ {#1} \right\}}

\new{\eps}{\epsilon}
\new{\ga}{\gamma}
\new{\Ga}{\Gamma}
\new{\la}{\lambda}
\new{\Imm}{{\rm ran}\,}
\new{\kappo}{\kappa^\prime}
\new{\muu}{\mu^\prime}

\new{\no}[1]{\left\|{#1}\right\|}
\new{\norh}[1]{\left\|{#1}\right\|_\hh}
\new{\nork}[1]{\left\|{#1}\right\|_\kk}
\new{\norq}[1]{\left\|{#1}\right\|_q}
\new{\norp}[1]{\left\|{#1}\right\|_p}
\new{\noru}[1]{\left\|{#1}\right\|}
\new{\normi}[1]{\left\|{#1}\right\|_{\infty}}
\new{\scal}[2]{\left\langle{#1},{#2}\right\rangle}
\new{\scalh}[2]{{\left\langle{#1},{#2}\right\rangle}_\hh}
\new{\scalhuno}[2]{{\left\langle{#1},{#2}\right\rangle}_{\hh_1}}
\new{\scalhdue}[2]{{\left\langle{#1},{#2}\right\rangle}_{\hh_2}}
\new{\scalk}[2]{{\left\langle{#1},{#2}\right\rangle}_\kk}
\new{\set}[1]{\{{#1}\}}
\new{\trh}[1]{{\rm tr}_{\mathcal{H}}\,{#1}}
\new{\trk}[1]{{\rm tr}_{\mathcal{K}}\,{#1}}

\new{\runo}{{\mathbb R}}
\new{\Z}{\mathbb{Z}}
\new{\cuno}{{\mathbb C}}
\new{\nat}{{\mathbb N}}

\new{\vv}{{\mathcal V}}
\new{\ff}{{\mathcal F}}
\new{\bb}{{\mathcal B}}
\new{\hh}{{\mathcal H}}
\new{\xx}{{\mathcal X}}
\new{\hhg}{{\hh_\ga}}
\new{\hhm}{{\hh_\mu}}
\new{\kk}{\mathcal Y}
\new{\elle}[1]{\mathcal{L}( #1 )}
\new{\BK}{{\mathcal{L}({\kk})}}
\new{\BH}{{\mathcal{L}(\mathcal{H})}}
\new{\BHdue}{{\mathcal{L}(\mathcal{H}_2)}}
\new{\BV}{{\mathcal{L}(\mathcal{V})}}
\new{\BHK}{{\mathcal{L}(\mathcal{H};{\kk})}}
\new{\BKH}{{\mathcal{L}({\kk};\mathcal{H})}}
\new{\BXK}{{\mathcal{L}(\mathcal{X};{\kk})}}
\new{\CXK}{{\mathcal{C}(X;{\kk})}}
\new{\CoXK}{{\mathcal{C}_0 (X;{\kk})}}
\new{\CoGK}{{\mathcal{C}_0 (G;{\kk})}}
\new{\CoXHdue}{{\mathcal{C}_0 (X;\hh_2)}}
\new{\Co}{\mathcal{C}_0}
\new{\Cb}{\mathcal{C}_b}

\new{\CbXK}{{\mathcal{C}_b(X;{\kk})}}
\new{\CZK}{{\mathcal{C}(Z;{\kk})}}

\new{\bo}[1]{{\mathcal B}(#1)}
\new{\cc}[1]{{\mathcal C}(#1)}
\new{\ccc}[1]{{\mathcal C}_c(#1)}
\new{\cccc}{\mathcal{C}}

\new{\luno}{L^1(X,\mu)}
\new{\lunol}{L^1(X,\de x)}
\new{\ldue}{L^2(X,\mu)}
\new{\lpi}{L^p(X,\mu)}
\new{\ldueG}{L^2(X,\mu)}
\new{\lpiG}{L^p(G,\mu)}
\new{\lpiprG}{L^{p/(p-1)}(G,\mu)}
\new{\lqu}{L^q(X,\mu)}
\new{\lri}{L^r(Y,\nu)}
\new{\lduet}{L^2(T,\rho)}
\new{\lunog}{L^1(G,\de g)}
\new{\ldueg}{L^2(G,\de g)}
\new{\LomuXK}{L^0(X,\mu; {\kk})}
\new{\LomuXC}{L^0(X,\mu; \mathbb{C})}
\new{\LpmuXK}{L^p(X,\mu; {\kk})}
\new{\LpmuZK}{L^p(Z,\mu; {\kk})}
\new{\LrmuXK}{L^r(X,\mu; {\kk})}
\new{\LqmuXK}{L^q(X,\mu; {\kk})}
\new{\LunomuXK}{L^1(X,\mu; {\kk})}
\new{\LduemuXK}{L^2(X,\mu; {\kk})}
\new{\LinfmuXK}{L^{\infty}(X,\mu; {\kk})}
\new{\LunolXK}{L^1(X,\de x ; {\kk})}
\new{\LduemuZK}{L^2(Z,\mu; {\kk})}
\new{\LunomuZK}{L^1(Z,\mu; {\kk})}
\new{\LinfmuZK}{L^{\infty}(Z,\mu; {\kk})}
\new{\LunoMK}{L^1(Z,|\M |; {\kk})}
\new{\LdueMK}{L^2(Z,|\M |; {\kk})}
\new{\LdueXMK}{L^2(X,|\M |; {\kk})}
\new{\LunoXMK}{L^1(X,|\M |; {\kk})}
\new{\LpXMK}{L^p(X,|\M |; {\kk})}
\new{\LqXMK}{L^q(X,|\M |; {\kk})}
\new{\LinfXMK}{L^{\infty}(X,|\M |; {\kk})}
\new{\LinfMK}{L^{\infty}(Z,|\M |; {\kk})}
\new{\LpconmuXK}{L^{\frac{p}{p-1}}(X,\mu; {\kk})}

\new{\LduemuK}{L^2(X,\mu; {\kk})}

\new{\maaa}{\mu\mathrm{-a.a.}}
\new{\lunot}{L^1(X)}
\new{\lduenot}{L^2(Y)}
\new{\ldueXnot}{L^2(X)}

\new{\supp}[1]{\operatornamewithlimits{supp}\,#1}
\new{\lin}[1]{\operatornamewithlimits{span}\{\,#1\}}
\new{\range}[1]{\operatornamewithlimits{Im}\,#1}
\new{\Ker}[1]{\operatornamewithlimits{ker}\,#1}
\new{\ran}[1]{\operatornamewithlimits{ran}\,#1}
\new{\argmin}{\operatornamewithlimits{argmin}}

\new{\mae}{\mu\mathrm{-a.e.}}
\new{\R}{\mathbb{R}}
\new{\M}{\mathsf{M}}
\new{\mmu}{\mathbf{\mu}}
\new{\llambda}{\mathbf{\lambda}}
\new{\boro}{\mathcal{B}(\Omega)}
\new{\borx}{\mathcal{B}(X)}
\new{\borz}{\mathcal{B}(Z)}

\new{\kkx}{\mathcal{F}(X;\kk)}

\newcommand{\real}{\mathbb R} 
\newcommand{\complex}{{\mathbb C}} 

\newcommand{\lft}{\left(} 
\newcommand{\rgt}{\right)}

\newcommand{\lfg}{\left\{} 
\newcommand{\rgg}{\right\}}

\new{\pair}[2]{\langle #1, #2\rangle}
\new{\conj}[1]{\overline{#1}}

\new{\ve}{K}
\new{\A}{A_\gamma}
\new{\F}[1]{{\mathcal F}(#1)}

\new{\nuh}{\hat{\nu}}
\new{\muh}{\hat{\mu}}
\new{\xh}{\hat{X}}
\new{\zh}{\hat{Z}}
\new{\E}{\kk}
\new{\LdueXE}{L^2(\xh,\nuh;\E)}
\new{\fh}{\hat{f}}
\new{\lah}{{\hat{\lambda}}}

\begin{document}
\setlength\arraycolsep{0pt}
\author{C.~Carmeli\thanks{C.~Carmeli,
DIFI, Universit\`a di Genova, and I.N.F.N.,
Sezione di Genova, Via Dodecaneso~33, 16146 Genova, Italy. e-mail:
carmeli@ge.infn.it}, , E. De Vito\thanks{E.~De Vito,
DSA., Universit\`a di Genova, 
Stradone S. Agostino 37, 16123 Genova, Italy, and I.N.F.N., Sezione di
Genova, Via Dodecaneso~33, 16146 Genova, Italy. e-mail: devito@dima.unige.it}, A.
Toigo\thanks{A.~Toigo,
DISI, Universit\`a di Genova, Via Dodecaneso 35, 16146 Genova, and I.N.F.N.,
Sezione di Genova, Via Dodecaneso~33, 16146 Genova, Italy. e-mail:
toigo@ge.infn.it}, V. Umanit\`a\thanks{V.~Umanit\`a,
DISI, Universit\`a di Genova, Via Dodecaneso 35, 16146 Genova, 
and Dipartimento di Matematica ``F. Brioschi'',
Politecnico di Milano, Piazza Leonardo da
Vinci 32, I-20133 Milano, Italy. e-mail:
veronica.umanita@polimi.it},
}
\title{Vector valued reproducing kernel Hilbert spaces and universality}
\date{\today}

\maketitle

\begin{abstract}
This paper is devoted to the study of vector valued reproducing kernel
Hilbert spaces. We focus on two aspects: vector valued feature maps
and universal 
kernels. In particular we characterize the structure of translation invariant
kernels on abelian groups and we relate it to the universality problem.
\end{abstract}

\section{Introduction}
In learning theory, reproducing kernel Hilbert spaces (RKHS)
are an important tool for designing learning algorithms, see for example
\cite{cusm02,smzh05,suzh08} and the book \cite{cuzh07}.
In the usual setting the elements of the RKHS are scalar
functions. The mathematical theory for scalar RKHS has been established in the seminal 
paper \cite{aro50}. For a standard reference see the book \cite{sai88}. 

In machine learning there is an increasing interest for
vector valued learning algorithms, see
\cite{leliwa01,evmipo05,cade07}. In this framework, the basic object
is a Hilbert space  of functions $f$ from a set $X$ into a normed vector
space $\kk$ with the property that, for any $x\in X$,  $\no{f(x)}\leq
C_x \no{f}$ 
for a positive constant $C_x$ independent of $f$.\\
The  theory of vector valued RKHS has been completely worked
out in the seminal paper \cite{sch64}, devoted to the characterization
of the Hilbert spaces that are continuously  embedded into a locally convex topological
vector space, see \cite{ped57}. In the case $\kk$ is itself a Hilbert
space, the theory can be simplified as shown in
\cite{mipo05,cadeto06,camipoyi08}. As in the scalar case, a RKHS
is completely characterize by a map 
$K$ from $X\times X$ into the space of bounded 
operators on $\kk$ such that
\[
\sum_{i,j = 1}^N \scal{K(x_i , x_j) y_j}{y_i} \geq 0
\]
for any $x_1,\ldots,x_N$ in $X$ and $y_1,\ldots,y_N$ in $\kk$.  Such a map
is called a $\kk$-reproducing kernel and the corresponding RKHS is denoted
by $\hh_K$.  
 
This paper focuses on three aspects of particular interest in vector
valued learning problems: 
\begin{itemize}
\item vector valued feature maps;
\item universal reproducing kernels;
\item translation invariant reproducing kernels.
\end{itemize}
The {\em feature map} approach is the standard way in which 
scalar RKHS are presented in learning theory, see for example \cite{scsm02}. A 
feature map is a function mapping the input space $X$ into an
arbitrary Hilbert space $\hh$ in such a way that $\hh$ can be identified with a
unique RKHS. Conversely, any RKHS can be realized as a closed subspace
of {\em a concrete
  Hilbert space}, called {\em feature space}, by means of a suitable
feature map -- typical examples of feature spaces are $\ell^2$ and
$L^2(X,\mu)$ for some measure $\mu$.\\ 
The concept of feature map is extended
to the vector valued setting in \cite{cadeto06,camipoyi08}, where a
feature map is defined as a function from $X$ 
into the space of bounded operators between $\kk$ and the feature
space $\hh$.\\
 In the first part of our paper, Section~\ref{tre} shows that sum,
product and {\em composition with 
  maps} of RKHS can be easily described by 
suitable feature maps. In particular we give an elementary proof of  Schur
lemma about the product of a scalar kernel with a vector valued
kernel. Moreover, we present several examples of vector valued RKHS, most of them considered
in \cite{mixuzh06,camipoyi08}. For each one of them we exhibit a nice feature space. This allows to describe the impact of these examples on some learning
algorithms, like the regularized least-squares \cite{evpopo00}.  

In the second part of the paper, Section~\ref{secuniversal}
discusses the problem of characterizing 
universal kernels. We say that  a $\kk$-reproducing kernel is {\em universal}
if the corresponding RKHS $\hh_K$ is dense in
$\LduemuXK$ for any probability measure 
$\mu$ on the input space $X$. This definition is motivated observing
that in learning theory the goal is to approximate a target function 
$f^*$ by means of a {\em prediction function} 
$f_n\in\hh_K$, depending on the data, in such a way the distance between 
$f^*$ and $f_n$ goes to zero when the number of data $n$ goes to infinity. 
In learning theory the ``right'' distance is given  by the norm in
$\LduemuXK$, where $\mu$ is the (unknown) probability distribution
modeling the sample of the input data, see \cite{cusm02}. The
possibility of learning any target function $f^*$ by means of
functions in $\hh_K$ is precisely the density of $\hh_K$ in
$\LduemuXK$. Since the probability measure $\mu$ is unknown, we
require that the above property holds for any choice of $\mu$ --
compare with the definition of {\em universal consistency} for a
learning algorithm \cite{gykokr02}. 
Under the condition that the elements of $\hh_K$ are continuous
functions vanishing at infinity, we prove that  universality of $\hh_K$ is equivalent to
require that $\hh_K$ is dense in $\CoXK$, the Banach space of
continuous functions vanishing at infinity with the uniform norm.
If $X$ is compact and $\hh=\cuno$, the density of $\hh_K$ in $\CoXK$
is precisely the definition of {\em universality}
given in \cite{ste01,zho03}. For arbitrary $X$ and
$\kk$, another definition of universality is suggested in \cite{camipoyi08} under the
assumption that the elements of $\hh_K$ are continuous functions. We show that
this last notion is equivalent to require  that $\hh_K$
is dense in $\LduemuXK$ for any probability measure $\mu$ with compact
support, or  that $\hh_K$ is dense in $\CXK$, the space of
continuous functions with the compact-open topology. If $X$ is not
compact, the two definitions of universality are not 
equivalent, as we show in two examples. To avoid confusion we refer to
the second notion as {\em compact-universality}. \\
We characterize both  universality and  compact-universality 
in terms of the injectivity of the integral operator
on $\LduemuXK$ whose kernel is  the reproducing
kernel $K$. For compact-universal kernels, this result is presented in a slightly different form in
\cite{camipoyi08} -- compare Theorem~\ref{univimplLKiniett} below with
Theorem~11 of \cite{camipoyi08}. However, our statement of the theorem
does not require a direct use of  vector valued
measures, our proof is simpler and it 
is based on the fact that any bounded linear functional $T$ on $\CoXK$ is of the
form
\[ T(f)=\int_X\scal{f(x)}{h(x)}\de\mu(x),\]
where $\mu$ is a probability measure and $h$ is a bounded measurable function
from $X$ to $\kk$ -- see Appendix~\ref{sec. mis. vett.}. 
Notice that, though in learning theory the
main issue is the density of the RKHS $\hh_K$ in $\LduemuXK$, however, our results hold if, in the
definition of universal kernels,
we replace $\LduemuXK$ with $\LpmuXK$ for any $1\leq p< \infty$. In particular, we show that $\hh_K$ is dense in $\CoXK$ if
and only if there exists $1\leq p< \infty$ such that $\hh_K$ is dense
in $\LpmuXK$ for any probability measure $\mu$. In that case,  $\hh_K$ is dense
in $\LqmuXK$ for any $1\leq q< \infty$.

In the third part of the paper, under the
assumption that $X$ is a group,  Section~\ref{secabelian} studies 
translation invariant reproducing kernels, that is, the kernels 
such that $K(x,t)=K_e(t^{-1}x)$ for some operator valued function
$K_e:X\to\BK$ of completely positive type. In particular, we show that any translation invariant
kernel is of the form
\[K(x,t)=A\pi_{x^{-1}t}A^*\]
for some unitary representation $\pi$ of $X$ acting on a Hilbert space $\hh$,
 and a bounded operator
$A:\hh\to\kk$. If $X$ is an abelian group,  SNAG theorem \cite{fol95} provides a more
explicit description of the reproducing kernel $K$, namely
\[K(x,t)=\int_{\xh}\chi(t-x) dQ(\chi),\]
where $\xh$ is the dual group and $Q$ is a positive operator valued
measure on $\xh$.  The above equation is precisely the content of Bochner theorem
for operator valued functions of positive type
\cite{boc59,faha72}. In particular, we show that the corresponding
RKHS $\hh_K$ can be always realized as a closed subspace of
$L^2(\xh,\nuh,\kk)$ where $\nuh$ is a suitable positive measure on $\xh$.
In this setting, we give a sufficient condition ensuring that 
a translation invariant kernel is universal. This condition is
also necessary if $X$ is compact or 
$\kk=\cuno$. For scalar kernels and compact-universality this result is
given in~\cite{mixuzh06}. We end the paper by discussing in
Section~\ref{secexample}  
the universality of some of the examples introduced in Section~\ref{tre}. 

\section{Background}\label{sez. not.}
In this section we set the main notations and we recall some basic
facts about vector valued reproducing kernels.
\subsection{Notations and assumptions}
In the following we fix a locally compact second countable topological
space $X$ and a complex separable Hilbert space $\kk$, whose norm and
scalar product are denoted by $\no{\cdot}$ and
$\scal{\cdot}{\cdot}$ respectively. Local compactness of $X$ is needed
in order to 
prove Theorem~\ref{semplice} in the appendix, which is at the root of
Theorem~\ref{main1}. The separability of $X$ and $\kk$ will avoid some
problems in measure theory. All these assumptions are always
satisfied in learning theory.

We denote by $\kkx$ the vector space of functions $f:X\to \kk$, by
$\CXK$ the subspace of continuous functions, and by $\CoXK$ the
subspace of continuous functions vanishing at infinity. If $\kk = \cuno$, we set
$\cccc(X)=\cccc(X;\cuno)$ and $\cccc_0 (X) =
\cccc_0(X,\cuno)$. If $X$ is compact, $\CoXK=\CXK$.  \\
We regard $\CXK$ as a locally convex
topological vector space by endowing it with the compact-open
topology\footnote{This 
  is the topology of uniform convergence on compact subsets defined by
  the family of seminorms $\no{f}_Z = \max_{x\in Z}\no{f(x)}$
for $Z$ varying over the compact subsets in $X$.} and $\CoXK$ as a 
Banach space with respect to the uniform norm $\no{f}_\infty =
\max_{x\in X}\no{f(x)}$.   

Let $\borx$ be the Borel $\sigma$-algebra of $X$. By a
{\em measure} on $X$ we mean a $\sigma$-additive map $\mu : \borx
\frecc [0,+\infty]$ which is finite on compact sets\footnote{Since $X$
  is locally compact second countable, then  $\mu$ is both
  inner and outer regular.}. We say that $\mu$ is 
a {\em probability measure} if $\mu (X) = 1$.   For $1 \leq
p < \infty$, $\LpmuXK$ denotes the Banach space of (equivalence
classes of) measurable\footnote{Since $\kk$ is separable, 
  measurability is equivalent to the fact that $\scal{f(\cdot)}{y}$ is
  measurable for all $y\in\kk $.} functions $f:X\to \kk$ such that $\no{f}^p$ is
$\mu$-integrable, with norm
${\no{f}_p = \left(\int_{X} \no{f(x)}^p \de\mu(x)\right)^{1/p}}$.
If $p=2$ we denote the scalar product in $L^2(X,\mu,\kk)$ by $\scal{\cdot}{\cdot}_2$.
For $p=\infty$, $\LinfmuXK$ is the Banach space of $\mu$-essentially
bounded measurable functions $f:X\to \kk$ with norm
$\no{f}_{\mu,\infty} = \mu {\rm -ess\, sup}_{x\in X}\no{f(x)}$.

If $\mu$ is a probability measure, clearly 
$$\CoXK \subset \LpmuXK\subset \LqmuXK$$ 
for all $1\leq q < p \leq \infty$, each inclusion being
continuous. Moreover, since $X$ is locally compact and second countable, 
$\CoXK$ is dense in $\LpmuXK$ for any $1\leq p<\infty$. 

If $\hh$ is an arbitrary (complex) Hilbert space we denote  its scalar product
by $\scal{\cdot}{\cdot}_\hh$ and its norm by
$\no{\cdot}_\hh$. When $\hh^\prime$ is another Hilbert space, we denote
by $\elle{\hh ; \hh^\prime}$ the 
Banach space of bounded operators from $\hh$ to
$\hh^\prime$ endowed with the uniform norm. In the case $\hh = \hh^\prime$, we
set  ${\mathcal L}(\hh)=\elle{\hh ; \hh}$.\\
Given $w_1,w_2\in\hh$, we let $w_1\otimes w_2$ be the rank one operator
\[ (w_1\otimes \overline{w_2})v=\scal{v}{w_2}_{\hh}w_1\qquad v\in\hh.\]

\subsection{Vector valued reproducing kernels}\label{seckernel}
We briefly recall the main properties of vector valued reproducing
kernel Hilbert spaces. Given $X$ and $\kk$ as above, a map $K :
X\times X \frecc \BK$ is called a {\em $\kk$-reproducing kernel} if 
$$
\sum_{i,j = 1}^N \scal{K(x_i , x_j) y_j}{y_i} \geq 0
$$
for any $x_1,\ldots,x_N$ in $X$, $y_1,\ldots,y_N$ in $\kk$ and $N\geq 1$.
Given $x\in X$, $\ve_x:\kk\to \kkx$ denotes the linear operator whose action on a
vector $y\in\kk$ is the function $\ve_xy\in\kkx$ defined by
\begin{equation}  \label{ern}
 (\ve_x y)(t) = K(t,x)y \qquad  t\in X. 
\end{equation}
Given a $\kk$-reproducing kernel $K$, there is a unique Hilbert space
${\hh_K\subset \kkx}$ satisfying 
\begin{align} 
&   \ve_x \in  \elle{\kk,\hh_K}\qquad\ \ x\in X \label{clemente}\\
&  f(x) =  \ve_x^\ast f \qquad\qquad x\in X,\  f\in\hh_K\label{silvio},  
\end{align}
where $\ve_x^\ast:\hh_K\to \kk$ is the adjoint of $\ve_x$, see Proposition~$2.1$ of
\cite{cadeto06}. The space $\hh_K$ is called the {\em reproducing kernel Hilbert 
space} associated with $K$, the corresponding scalar product and  norm
are denoted by $\scal{\cdot}{\cdot}_K$ and $\no{\cdot}_K$, respectively.
As a consequence of\eqn{silvio}, we have that  
\begin{align*}
&  K(x,t)\,  =\,   \ve_x^\ast \ve_t \qquad x,t\in X 
\\
& \hh_K \, = \,\Spannochiuso{\ve_x y \mid x\in X, y\in\kk}. 
\end{align*}
As discussed in the introduction, the space $\hh_K$ can be realized as
a closed subspace of some arbitrary Hilbert space by means of a
suitable feature map, as shown by the next result, see Proposition~2.4 of~\cite{cadeto06}. 
\begin{proposition}\label{featuremap} 
Let $\hh$ be a Hilbert space and
  $\gamma:X\frecc\mathcal{B}({\kk};\mathcal{H})$. Then the
  operator $W:{\hh}\frecc\kkx$ defined by
\begin{equation}\label{defW} 
(W u)(x)=\gamma_x^*u , \qquad u\in\hh,\ x\in X,
\end{equation}
is a partial isometry from ${\hh}$ onto the reproducing kernel Hilbert
space $\hh_K$ with reproducing kernel
\begin{equation}
  \label{kernel_feature}
K(x,t)=\gamma_x^*\gamma_t, \qquad x,t\in X .
\end{equation}
Moreover, $W^* W$ is the orthogonal projection onto 
\[\ker{W}^\perp=\Spannochiuso{ \gamma_x y \mid x\in X,\ y\in\kk},\]
and
\begin{equation*}
\no{f}_K=\inf\set{\no{u}_\hh\mid u\in\hh,\ Wu=f}.
\end{equation*}
\end{proposition}

The map $\ga$ is usually called the {\em feature map}, $W$ the {\em feature operator} and 
 $\hh$ the {\em feature space}. Since $W$ is an isometry from
 $\ker{W}^\perp$ onto $\hh_K$, the map $W$ allows us to identify $\hh_K$
 with the closed subspace $\ker{W}^\perp$ of $\hh$.   
With a mild abuse of notation, we say that {\em $\hh_K$ is embedded
into $\hh$ by means of the feature operator $W$}.\\
Comparing\eqn{defW} with\eqn{silvio}, we notice that any RKHS $\hh_K$
admits a trivial feature map, namely $\ga_x=K_x$. In this case the feature
operator is the identity. Conversely, if $\hh$ is a Hilbert space of
functions from $X$ to $\kk$ such that $\no{f}\leq C_x\no{f}_{\hh}$ for
some positive constant $C_x$, then there exists a bounded operator
$\ga_x:\kk\to\hh$ such that $f(x)=\ga_x^\ast f $. Hence, the above proposition
implies that $\hh$ is a RKHS with kernel 
given by\eqn{kernel_feature} and that the feature operator is the identity.

\subsection{Mercer and $\Co$-kernels}
In this paper, we mainly focus on reproducing
kernel Hilbert spaces, whose elements are continuous functions.
In particular we study the following  two classes of reproducing kernels.
\begin{definition} A reproducing kernel $K:X\times X\to\BK$ is called
  \begin{enumerate}
  \item[{\rm (i)}]\label{KerneldiMercer}
   {\em Mercer} provided that  $\hh_K$ is a subspace of $\CXK$;
   \item[{\rm (ii)}]\label{cokernel}
   {\em $\Co$} provided that $\hh_K$ is a subspace of $\CoXK$.
  \end{enumerate}
\end{definition}
The choice of $\CXK$ and $\CoXK$ is motivated in
Section~\ref{secuniversal} where we discuss the
universality problem.

The following proposition directly characterizes Mercer
and $\Co$-kernels in terms of properties of the kernels. 
\begin{proposition}\label{equivalenza in Co}
Let $K$ be a reproducing kernel.
\begin{itemize}
\item[{\rm (i)}] The kernel $K$ is Mercer iff the function
  $x\longmapsto\no{K(x,x)}$ is locally bounded and 
  $\ve_x y \in \CXK$ for all $x\in X$ and $y\in\kk$.
\item[{\rm (ii)}] The kernel $K$ is $\Co$ iff the function $x\longmapsto\no{K(x,x)}$ is bounded
 and $\ve_x y \in \CoXK$ for all $x\in X$ and $y\in\kk$.
\end{itemize}
If $K$ is a Mercer kernel, the inclusion $\hh_K \hookrightarrow \CXK$ is 
continuous. If $K$ is a $\Co$-kernel the inclusion $\hh_K
\hookrightarrow \CoXK$ is continuous. In both cases, the space $\hh_K$ is separable. 
\end{proposition}
\begin{proof}
We prove only~(ii), since the other proof is similar -- see 
Proposition~5.1 of~\cite{cadeto06}. If $\hh_K \subset \CoXK$, it is clear that $\ve_x
y$ is an element of $\CoXK$. Moreover, since $\no{\ve_x^\ast f} =
\no{f(x)} \leq \no{f}_\infty$ $\forall f\in \mathcal{H}_K$, by the
principle of uniform boundedness 
there exists $M<\infty$ such that $\no{{\ve_x^\ast}} \leq M$ for all
$x$. Therefore, $\no{K(x,x)} = \no{{\ve_x^\ast}}^2 \leq M^2$ for all $x$.\\
Conversely,  assume that the function $x\longmapsto\no{K(x,x)}$ is bounded
 and $\ve_x y \in \CoXK$. Given $f\in\hh_K$, we have
\[\no{f(x)}\leq \no{f}_{K} \no{K(x,x)}^{1/2} \leq M \no{f}_{K}\,.\]
In particular, convergence in $\hh_K$ implies uniform
convergence, so that the closure (in $\hh_K$) of the linear span of 
$\set{\ve_x y \mid x\in X, y\in\kk}$
is contained in  $\CoXK$, {\em i.e.} $\hh_K\subseteq \CoXK$.\\
The continuity of the inclusion of $\hh_K$ in $\CoXK$ 
follows from $\no{f}_\infty \leq M \no{f}_{\hh_K}$. Finally
$\hh_K$ is separable by Corollary~5.2 of~\cite{cadeto06}. 
\end{proof}
If $\hh_K$ is defined by means of a feature map $\ga$, the above
characterization can be expressed in terms of $\ga$, as shown by the
following result.
\begin{corollary}\label{featureco}
With the notations of Proposition~$\ref{featuremap}$ the following
conditions are equivalent.
\begin{enumerate}
\item[{\rm (a)}] The kernel $K$ is Mercer $[$resp. $\Co]$.
\item[{\rm (b)}] There is a total set $\mathcal S$ in $\hh$ such that
  $W({\mathcal S})\subset\CXK$ $[$resp.  $W({\mathcal S})\subset\CoXK]$
  and the function $x\longmapsto\no{\ga_x}$ is locally bounded
  $[$resp. bounded$]$.  
\end{enumerate}
\end{corollary}
\begin{proof}
We give the proof only in the case of a $\Co$-kernel, the other case being simpler.
Suppose hence (a) holds true, {\em i.e.} $\hh_K \subset\Co$, then $W(\mathcal
S)\subset\ran{W}=\hh_K\subset\CoXK$ for all subset $\mathcal{S}$ of $\hh$. Moreover,
$\no{\ga_x}^2=\no{K(x,x)}\leq M$
by item~(ii) of Proposition~$\ref{equivalenza in Co}$. Conversely, if
condition (b) holds, we have that for all $x\in X$ and $u\in
\hh$ 
\[\no{(Wu)(x)}=\no{K_x^\ast( Wu)}\leq
  \no{K_x^\ast}\no{W}\no{u}_\hh\leq\no{K(x,x)}^{\frac{1}{2}}\no{u}_\hh\leq
  M^{\frac{1}{2}}\no{u}_\hh,\] 
where $\no{W}\leq 1$ being  $W$ a partial isometry.
Then $W$ maps $\hh$ into the space of bounded functions and $W$ is
continuous from $\hh$ onto $\hh_K$ endowed with the uniform
norm. Since $W(\mathcal S)\subset \CoXK$ and $\CoXK$ is complete, then
$\hh_K\subset\CoXK$.      
\end{proof}
\subsection{Mercer theorem}\label{submercer}
For a Mercer kernel $K$, there is a {\em canonical} feature map, based on
Mercer theorem, which relates the spectral properties of the
integral operator with kernel $K$, and the structure of the
corresponding reproducing kernel Hilbert space. This result will be
also used in the examples. \\ 
To state this result for vector valued
reproducing kernels, we need some preliminary facts.
First of all,  if $K$ is a Mercer kernel and $\mu$ is a probability measure
on $X$, the space $\hh_K$ is a subspace 
of $\LduemuXK$, provided that $\no{K(x,x)}$ is bounded on the support of $\mu$. 
This last condition is always satisfied if $K$ is a $\Co$-kernel or if
$\mu$ has compact support. If $\hh_K$ is a subspace of $\LduemuXK$, we
denote the canonical inclusion by  
\[i_{\mu}:\hh_K\hookrightarrow\LduemuXK.\] 
Next lemma states some properties of $i_{\mu}$ and its proof  
is a consequence of Propositions~$3.3$,~$4.4$ and~$4.8$ of~\cite{cadeto06}. 
\begin{proposition}\label{integral} 
Let $K$ be a Mercer kernel and $\mu$ a probability
measure such that $K$ is bounded on the support of $\mu$. The
inclusion $i_{\mu}$ is a bounded operator, 
its adjoint $i_\mu^*:\LduemuXK\frecc\hh_K$ is given by
\begin{equation*}
(i_\mu^\ast f)(x) = \int_{X} K(x,t)f(t)\de\mu(t),
\end{equation*}
where the integral converges in norm, and the composition
$ i_\mu i_\mu^*=L_\mu$ is the
integral operator on $\LduemuXK$ with kernel $K$ 
\begin{equation*}
(L_\mu f)(x) = \int_{X} K(x,t)f(t)\de\mu(t).
\end{equation*}
In particular, if $K(x,x)$ is a compact operator for
all $x\in X$, then $L_K$ is a compact operator.
\end{proposition}
The fact that $L_K$ is a compact operator implies that there is a
family $(f_i)_{i\in I}$ of eigenvectors in $\CXK$ and a family
$(\sigma_i)_{i\in I}$ of eigenvalues in $]0,\infty[$ such that $(f_i)_{i\in I}$
is an orthonormal basis of $\ker{L_\mu}^\perp=\overline{\ran{L_\mu}}$ and
\begin{equation}\label{spectral}
L_\mu f_i  =  \sigma_if_i.
\end{equation}
With this notation we are ready to state Mercer Theorem for vector
valued kernels. Its  proof is consequence of Proposition~6.1 and
Theorem~6.3 of~\cite{cadeto06}. 
\begin{proposition}\label{mercer}
Let  $\mu$ be a probability measure with $\supp{\mu}=X$. Suppose  $K$ is 
a Mercer kernel such that $\sup_{x\in X}\no{K(x,x)}<\infty$, and $K(x,x)$ is 
a compact operator $\forall x\in X$. With the notation
of$\eqn{spectral}$, we have that
\begin{align}
   &\hh_K =  \set{f\in\CXK\cap \ker{L_\mu}^\perp\mid \sum_{i\in I}
    \frac{|\scal{f}{f_i}_2|^2}{\sigma_i}<\infty} \label{mercer2}\\
   &\scal{f}{g}_K  =  \sum_{i\in I}
    \frac{\scal{f}{f_i}_2\scal{f_i}{g}_2}{\sigma_i}\label{mercer3}\\
\label{mercer1}
  & K(x,t) = \sum_{i\in I} \sigma_i f_i(x) \otimes \overline{f_i(t)}
\end{align}
where the last series converges in the strong operator topology of $\BK$.
\end{proposition}
Equations\eqn{mercer2} and\eqn{mercer3} imply that
$(\sqrt{\sigma_i}f_i)_{i\in I}$ is an orthonormal basis in
$\hh_K$. In particular the vectors $\sqrt{\sigma_i}f_i$
are $\ell_2$-linearly independent in $\kkx$, namely, if
$(c_i)_{i\in I}$ is a family such that $\sum_{i\in I}|c_i|^2<\infty$ and
$\sum_{i\in I} c_i\sqrt{\sigma_i}f_i(x)=0$ for all $x\in X$, then  $c_i=0$ for all
$i\in I$.

As said at the beginning of Section~\ref{submercer}, Proposition~\ref{mercer}
gives a feature operator, which is often used in learning theory.
\begin{example}\label{ex_mercer}
With the assumptions and notations of Proposition~$\ref{mercer}$, 
the reproducing kernel Hilbert space $\hh_K$ is unitarily equivalent
to $\ker{L_\mu}^\perp=\overline{\ran{L_\mu}}$ by means of the
feature operator 
\begin{equation}
\label{Wex1}
(Wf)(x)=\sum_{i\in I}\sqrt{\sigma_i}f_i(x)\scal{f}{f_i}_2=
(L_\mu^{\frac 12}f)(x)\,,\quad f\in \LduemuXK\,.
\end{equation}
\end{example}
\begin{proof}
Given $x\in X$, define
\[ \ga_x:\kk\to\LduemuXK\qquad \ga_x y = \sum_{i\in
  I}\sqrt{\sigma_i}\scal{y}{f_i(x)}f_i,\]
which is well defined since $(f_i)_{i\in I}$ is orthonormal
family of continuous functions and\eqn{mercer1} ensures that
$\sum_{i\in I}\sigma_i|\scal{y}{f_i(x)}|^2<\infty$. Using\eqn{mercer1}
again, one checks that $\ga_x^\ast\ga_t=K(x,t)$. The fact that 
feature operator is given by\eqn{Wex1} is clear by definition of
$\ga_x$. 
Since $\ker{W}=\ker{L_\mu}$, $W$ is a unitary operator from
$\ker{L_\mu}^\perp$ onto $\hh_K$.    
\end{proof}

\subsection{Trivial examples}
We give two examples of {\em trivial} vector valued kernels.
\begin{example}\label{operatore}
Let $B\in\BK$ be a positive operator and define $K(x,t)=B$ for all $x,t\in
X$,  then $K$ is a $\kk$-reproducing kernel, $\hh_K$ is unitarily
equivalent to $\ker{B}^\perp=\overline{\ran{B}}$ by means of the
feature operator 
\[ (Wy)(x)=B^{\frac12}y\qquad x\in X,\ y\in\ker{B}^\perp.\]
The kernel $K$ is of Mercer type and it is a $\Co$-kernel if and only
if $X$ is compact.
\end{example}
\begin{proof}
Apply Proposition~$\ref{featuremap}$ with $\hh=\ker{B}^\perp$ and
$\ga_x=B^\frac{1}{2}$. Since $B$ is injective on $\hh$, then $W$ is
unitary.  The claims about the continuity are clear.
\end{proof}
\begin{example}\label{uno}
Let $f:X\to \kk$, $f\neq 0$.
Define $K(x,t)= f(x)\otimes\overline{f(t)}$, then $K$ is a
reproducing kernel, $\hh_K$ is unitarily equivalent to $\complex$ by
means of the feature operator
\[ (W c)(x)=cf(x) \qquad x\in X,\ c\in\complex.\]
In particular $K$ is Mercer $[$resp.$\Co]$ if and only if $f\in\CXK$
$[$resp. $f\in\CoXK]$.
\end{example}
\begin{proof}
Apply Proposition~$\ref{featuremap}$ with $\hh=\complex$ and
$\ga_xy=\scal{y}{f(x)}$. Since $f\neq 0$, $W$ is injective.
The characterization about Mercer and $\Co$ is trivial.
\end{proof}

\section{Operations with kernels}\label{tre}
In this section we characterize reproducing kernel Hilbert spaces
whose kernel is defined by {\em algebraic operations}, like sum,
product and composition. Most of the results are well known for scalar
kernels, whereas for vector valued kernels they are consequences of the
theory developed in \cite{sch64} in a more general context. We provide
a direct and simple proof 
of these results,  based on
the use of suitable feature maps. In some cases, our approach can be
of  interest 
also in the scalar case, like, for example, in proving  Schur lemma
about the product of kernels.  \\
As an application, we present a large supply of examples of vector
valued reproducing kernels and, for most of them, we realize the
corresponding RKHS by elegant and simple structures. This
characterization will be used to analyze some learning algorithm, like
regularized least-squares, in the vector valued setting.

\subsection{Sum of kernels}
The following result extends to vector valued kernels the relation between
sum of kernels and sum of the corresponding reproducing kernel Hilbert
spaces.
\begin{proposition}\label{suma}
Denote by $I$ a countable set and let $(K^i)_{i\in I}$ be a family of $\kk$-reproducing kernels such that 
\[\sum_{i\in I} \scal{K^i(x,x)y}{y}<\infty\qquad \forall y\in \kk\text{
  and }\forall x\in X.\] 
Given $x,t\in X$ the series $\sum_{i\in I} K^i(x,t)$ converges to a
bounded operator $K(x,t)$ in the strong operator topology, and the map
$K:X\times X\to\BK$ defined by
\[ K(x,t)y=\sum_{i\in I} K^i(x,t)y \]
is a $\kk$-reproducing kernel. The corresponding space $\hh_K$ is embedded in
$\bigoplus_{i\in I}\hh_{K^i}$ by means of the feature operator
\[ W(f)(x)=\sum_{i\in I} f_i(x)\qquad\text{where } f=\oplus_{i\in I} f_i\]
where the sum converges in norm.\\
Moreover, if each $K^i$ is a Mercer kernel $[$resp. $\Co$-kernel$]$  and
$x\mapsto \sum_{i\in 
  I}\no{K^i(x,x)}$ is locally bounded $[$resp. bounded$]$, then $K$ is
Mercer  $[$resp. $\Co]$.
\end{proposition}
\begin{proof}
We apply Proposition~\ref{featuremap}.
Letting $\hh=\bigoplus_{i\in I}\hh_{K^i}$, we regard each
$\hh_{K^i}$ as a closed subspace of $\hh$ so that any two of them are
orthogonal. Given $x\in X$,  we define  
the bounded operator $\ga_x:\kk\to\hh$
by $\ga_x=\sum_{i\in I}K^i_x$, where the series converges in the strong
operator topology since, given $y\in\kk$,
\[ \sum_{i\in I} \no{K^i_xy}_{K^i}^2 =\sum_{i\in I}\scal{K^i(x,x)y}{y}<\infty\]
by assumption, see \cite{Conway}.  Given $i\in I$ and  $f_i\in\hh_{K^i}$, then
\[\scal{\ga_x^\ast
  f_i}{y}=\scal{f_i}{K^i_xy}_{K^i}=\scal{f_i(x)}{y}\]   
by reproducing property\eqn{silvio}, so that $\ga_x^\ast
f_i=f_i(x)$. Since $\ga_x^\ast$ is continuous, 
for any $f=\oplus_{i\in I} f_i$, 
\[(Wf)(x)=\ga_x^\ast f=\sum_{i\in I}\ga_x^\ast f_i=\sum_{i\in I} f_i(x)\]
where the series converges in norm.\\ Finally,  
$K(x,t)y=\ga_x^\ast\ga_t y=\sum_{i\in I}(\ga_t y)_i(x)=\sum_{i\in
  I} K^i(x,t)y$, that is $K(x,t)=\sum_{i\in
  I} K^i(x,t)$ in the strong operator topology.\\
The second part is a consequence of Corollary~\ref{featureco} with
$\mathcal S =\bigcup_{i\in I}\hh_{K^i}$.
\end{proof}
As an application, we have the following example.
\begin{example}
Let $(f_i)_{\in I}$ a countable family of functions $f_i:X\to \kk$ such
that $\sum_{i\in I} |\scal{f_i(x)}{y}|^2$ is finite for all $x\in X$
and $y\in Y$. Define $K:X\times X\to\BK$ as
\begin{equation*}
 K(x,t)=\sum_{i\in I} f_i(x)\otimes \overline{f_i(t)}\,.
\end{equation*}
Then, the sum converges in the strong operator topology, $K$ is a
reproducing kernel and  
\begin{equation}
  \label{4}
\hh_K=\set{f\in\kkx\mid f(x)=\sum_{i\in I}c_i f_i(x), \sum_{i\in I}
  |c_i|^2<\infty}.
\end{equation}
In particular $(f_i)_{i\in I}$ is a normalized tight frame in
$\hh_K$. It is an orthonormal basis if and only if  $(f_i)_{i\in I}$ is
$\ell_2$-linearly independent in $\kkx$. 
\end{example}
\begin{proof}
Apply Proposition~$\ref{suma}$, with 
$K^i(x,t)=f_i(x)\otimes \overline{f_i(t)}$, observing that
$\hh_{K^i}=\complex$ by Example~$\ref{uno}$,  so that $\oplus_{i\in
  I}\hh_{K_i}\simeq\ell_2$. The feature operator is explicitly given by
\[ W(c)(x)=\sum_{i\in I} c_i f_i(x)\qquad\text{where }
c=(c_i)_{i\in I},\ \sum_{i\in I} |c_i|^2<\infty,\]
so that\eqn{4} is clear. If $(e_i)_{i\in I}$ is the canonical
orthonormal basis
of $\ell_2$, then $We_i=f_i$ and, for any $f\in\hh_K$,
$$  
\no{f}^2_K= \no{W^\ast f}^2_{\ell_2}= \sum_i |\scal{W^\ast
  f}{e_i}_{\ell_2}|^2 =\sum_i |\scal{f}{f_i}_K|^2\,, 
$$
{\em i.e.} $(f_i)_{i\in I}$ is a normalized tight frame in
$\hh_K$. Clearly, it is an orthonormal basis if and
only if $W$ is unitary, {\em i.e.} W is injective. This is precisely
the condition that $(f_i)_{i\in I}$ is $\ell_2$-linearly independent
in  $\kkx$.   
\end{proof}
Proposition~\ref{mercer} shows that any RKHS with a bounded compact
Mercer kernel can be realized 
as in the above example, where the functions $f_i$ are the
eigenfunctions (with $\no{f_i}_2^2=\sigma_i$) of the integral operator $L_\mu$
with eigenvalues $\sigma_i>0$, and $\mu$ is any probability measure
with $\supp{\mu}=X$,  see\eqn{spectral}. 

\subsection{Composition with maps}
We now describe the reproducing kernel Hilbert spaces whose kernel is
defined in terms of a {\em mother} kernel and suitable maps acting
either on the input  
space $X$  or on the output space $\kk$. The following result
characterizes the action of a bounded operator on $\kk$. 
\begin{proposition}\label{actionw}
Let $K$ be a $\kk$-reproducing kernel. 
Let $\kk^\prime$ be another Hilbert space and $w:\kk\to\kk^\prime$ be a
bounded operator.  Define 
\[K_w:X\times X\to{\mathcal L}(\kk^\prime)\qquad
K_w(x,t)=wK(x,t)w^\ast,\] 
then $K_w$ is a $\kk^\prime$ reproducing kernel and $\hh_{K_w}$ is embedded in 
$\hh_K$ by means of
\[ W:\, \hh_K \longrightarrow \hh_{K_w}\,,\quad (Wf)(x)=w f(x)\qquad x\in X.\] 
If $w$ is injective, $\hh_{K_w}$ is unitarily equivalent to $\hh_K$.
Moreover, if $K$ is Mercer $[$resp. $\Co$$]$, then $K_w$ is Mercer
$[$resp. $\Co$$]$. 
\end{proposition}
\begin{proof}
Let $\ga_x:\kk^\prime\to\hh_K$, $\ga_x=K_xw^\ast $ and apply
Proposition~\ref{featuremap} with 
$\hh=\hh_K$. The feature operator from $\hh_K$ onto $\hh_{K_w}$ is
explicitly given by $(Wf)(x)=\ga_x^\ast f=wf(x)$. If $w$ is injective,
then $W$ is unitary.  The second claim is evident. 
\end{proof}
We now study the action of an arbitrary map on $X$.
\begin{proposition}\label{olanda}
Let $K$ be a $\kk$-reproducing kernel on $X$.
Let $T$ be another locally compact second countable topological space,
and  $\Psi:T\to X$. Define 
\[K_\Psi:T\times T\to\BK\qquad K_\Psi(t_1,t_2)=K(\Psi(t_1),\Psi(t_2))\qquad t_1,t_2\in T.\]
Then $K_\Psi$ is a $\kk$-reproducing kernel on $T$, the space $\hh_{K_\Psi}$ is
unitarily equivalent to  
\begin{equation*}
\begin{split}
 \Spannochiuso{K_x y\mid x\in\ran{\Psi}} & =
\set{f\in\hh_K \mid f(x)=0 \ \forall x\in\ran{\Psi}}^\perp 
\end{split}
 \end{equation*}
by means of the feature operator
\[W:\, \hh_K \longrightarrow \hh_{K_\Psi} \qquad W(f)(t) =
f(\Psi(t))\qquad f\in\hh_K,\ t\in T.\] 
If $K$ is a Mercer kernel and $\Psi$ is continuous, then $K_{\Psi}$ is
Mercer. If $K$ is a $\Co$-kernel and $\Psi$ is continuous and proper, then $K_\Psi$
is $\Co$.
\end{proposition}
\begin{proof}
Apply Proposition~\ref{featuremap} with $\hh=\hh_K$ and, for any $t\in T$, 
$\ga_t=K_{\Psi(t)}$, observing that $\ker{W}=\set{f\in\hh_K \mid
  f(x)=0 \ \forall x\in\ran{\Psi}}$.  \\ 
The claims about Mercer and $\Co$-kernels are clear.
\end{proof}

In the above proposition observe that $\ker{W}^\perp$ can be
identified with the quotient 
space $\hh_K/\ker{W}$, so that one has also the  natural identification
\begin{align}
\label{winter}
\hh_{K_\Psi}
  \simeq \set{f_{\mid{\ran{\Psi}}}\mid f\in \hh_K}  
\end{align}
where, the r.h.s.~is endowed with the norm
\[
\no{f_{\mid\ran{\Psi}}}=\inf \set{\no{g }_{K}\mid
  g\in\hh_K,\ g_{\mid{\ran{\Psi}}}=f_{\mid\ran{\Psi}}}
\]

As a consequence,  we describe the
relation between a kernel and its restriction to a subset.
\begin{corollary}\label{HZ}
Let $X_0$ be a subset of $X$. Let $K_{X_0}$ be the restriction of
$K$ to $X_0\times X_0$, then
\begin{equation*}
\hh_{K_{X_0}}= \set{f_{\mid_{X_0}}\ :\ f\in \hh_K}.
\end{equation*}
If $K$ is Mercer and $X_0$ is locally closed, then $K_{X_0}$ is Mercer.
If $K$ is $\Co$ and $X_0$ is closed, then $K_{X_0}$ is $\Co$.
\end{corollary}
\begin{proof}
Apply  Proposition~\ref{olanda} and identification (\ref{winter}), with $\Psi$  the canonical
inclusion of $X_0$ in $X$.  
\end{proof}
We end this part by describing the reproducing kernel Hilbert space
associated with the kernel proposed in \cite{camipoyi08}.
\begin{proposition}\label{ex_psi}
Let $\kappa$ be a scalar reproducing kernel on $X$.
Let $T$ be another locally compact second countable
topological space. Let $\Psi_1,\ldots,\Psi_m$ be functions from $T$ to
$X$ and define $K(t_1,t_2)$ as the $m\times m$-matrix
\[ K(t_1,t_2)_{ij}=\kappa(\Psi_i(t_1),\Psi_j(t_2))\qquad i,j=1,\dots,m,\
t_1,t_2\in T.\]
Then $K$ is a $\complex^m$-reproducing kernel on $T$, the space $\hh_K$ is
embedded in $\hh_\kappa$ by means
of the feature operator
\[
W:\, \hh_\kappa \longrightarrow \hh_K \qquad 
\left(W(\varphi)(t)\right)_i = \varphi(\Psi_i(t))\qquad \varphi\in\hh_\kappa,\
t\in T.\]
If one of $\Psi_1$, $\ldots$, $\Psi_m$ is surjective, then $W$ is
unitary.
\end{proposition}
\begin{proof}
Apply Proposition~\ref{featuremap} with $\hh=\hh_\kappa$ and 
$\ga_t:\complex^m\to \hh_\kappa$,
\[\ga_t(y_1,\ldots,y_m)=\sum_{i=1}^m y_i \kappa_{_{\Psi_i(t)}},\]
so that $\ga_t^\ast(\varphi)_i= \varphi(\Psi_i(t))$. \\
If $\Psi_i$ is surjective for some index $i=1,\dots, m$, the condition
$\varphi(\Psi_i(t))=0$ for all 
$t\in T$ implies that $\varphi(x)=0$ for all $x\in X$, that is, $\varphi=0$. Hence
$W$ is injective and, hence, unitary.        
\end{proof}

\subsection{Product of kernels}
The following proposition extends Schur lemma about
products of reproducing kernels  to the vector valued case.
\begin{proposition}\label{product}
Let $K$ be a $\kk$-kernel and  $\kappa$  a scalar  kernel.
 Define 
\[(\kappa K)(x,t)=\kappa(x,t)K(x,t)\qquad x,t\in X,\]
then $\kappa K$ is a  $\kk$-reproducing
kernel and  $\hh_{\kappa K}$ is
embedded into $\hh_\kappa\otimes \hh_K$ by means of the
feature operator
\[ W(\varphi\otimes f)(x)=\varphi(x) f(x)\qquad\varphi\in\hh_\kappa,\ 
f\in \hh_K.\]
If both $\kappa$ and $K$ are Mercer kernels, so is $\kappa K$, whereas if
\begin{equation}\label{productco}
\sup_{x\in X}\set{\kappa(x,x),\no{K(x,x)}}<\infty\text{ and }
\left\{\begin{array}{ccc}
\kappa_x\in\cccc_0(X) &\text{ and } & K_xv\in\CXK \\  
& \text{or} & \\
\kappa_x\in\cccc(X) & \text{and} & K_xv\in\CoXK
\end{array} \right.
\end{equation}
then $K$ is a $\Co$ kernel.
\end{proposition}
\begin{proof}
Let $\hh=\hh_\kappa\otimes \hh_K$. Since $\kappa$ is a scalar kernel,
$\kappa_x\in\hh_\kappa$.  
Define $\ga_x:\kk\to\hh$ by means of $\ga_xy=\kappa_x\otimes K_xy$, then
$\ga_x^\ast(\varphi\otimes f)=\varphi(x) f(x)$.  First claim is a consequence
of Proposition~\ref{featuremap}.\\  
If both $\kappa$ and $K$ are Mercer kernels, clearly $\kappa K$ is Mercer.\\
To prove that if\eqn{productco} hold then $K$ is $\Co$, we apply
Corollary~\ref{featureco} with 
$\mathcal S= \set{\varphi\otimes f\mid 
  \varphi\in\hh_\kappa,\ f\in\hh_K}$, and  observe that
\[\no{\ga_x}\leq\no{\kappa_x}_{\kappa}\no{K_x}\leq C,\] 
by assumption.
\end{proof}
Based on the above results, we characterize the RKHS whose kernel is
given in \cite{camipoyi08}. 
\begin{example}\label{esempioB}
Let $\kappa$ be a scalar reproducing kernel and $B$ a positive bounded
operator on $\kk$. Define $K:X\times X\to\BK$ as
\[K(x,t)=\kappa(x,t) B\qquad x,t\in X\] 
\begin{enumerate}
\item[{\rm (i)}] The map $K$ is a
$\kk$-reproducing kernel and $\hh_K$ is unitarily
equivalent to $\hh_\kappa\otimes \ker{B}^\perp$ by means of the unitary
operator 
\[W(\varphi\otimes y)(x)=\varphi(x) B^{\frac 12}y .\] 
\item[{\rm (ii)}] If $\kappa$ is  Mercer $[$resp. $\Co]$, then $K$
  is  Mercer $[$resp. $\Co]$, too. 
\item[{\rm (iii)}] If there is an orthonormal basis $(y_i)_{i\in I}$  of $\ker{B}^\perp$
such that $By_i=\sigma_i y_i$ (so that $\sigma_i>0$ for all $i\in I$), then $\hh_K$ is unitarily
equivalent to $\oplus_{i\in I} \hh_\kappa$ by means of the
unitary operator 
\begin{equation}
\label{wtilde}
\widetilde{W}(\oplus_{i\in I} \varphi_i)(x)=\sum_{i\in
  I} \sqrt{\sigma_i}\varphi_i(x) y_i\,,
\end{equation}
where the series converges in norm.
\end{enumerate}
\end{example}
\begin{proof}
First two items are a consequence of Proposition~$\ref{product}$ and
Example~$\ref{operatore}$. We prove item~(iii) in two steps. 
Apply first Proposition~$\ref{actionw}$ with $w:\kk\to\ell_2$,
$(wy)_i=\scal{y}{y_i}$, so that $\hh_{K_w}$ is embedded in 
$\hh_K$, by means of the feature operator $W_w( f) = w\circ f$
for all $f\in \hh_K$. The corresponding $\ell_2$-kernel is
$K_w(x,t)=\kappa(x,t)w Bw^\ast$.
By definition of $w$, the kernel $K_w$ is diagonal with respect to  $(e_i)_{i\in
  I}$, the canonical basis of $\ell_2$, namely
\[K_w(x,t)= \sum_{i\in I}\sigma_i\kappa(x,t) e_i\otimes
\overline{e_i} =:\sum_{i\in I} K^i(x,t), \]
where the series converges in the strong operator topology. \\
Now observe that, for each $i\in I$, ${\rm ker} (\sigma_ i e_i\otimes
\overline{e_i})^\perp= \complex e_i$, so that for item~(i) of this
example, the space $\hh_{K^i}$ is
unitarily  equivalent to $\hh_\kappa\otimes \complex \, e_i\simeq
\hh_\kappa$, through the feature operator 
\[
W^i:\hh_\kappa \rightarrow  \hh_{K^i}\,, \qquad  W^i(\varphi) (x) =
\varphi(x) \sqrt{\sigma_i }e_i 
\]
Applying Proposition~$\ref{suma}$ to the family $(K^i)_{i\in I}$,  we obtain 
 a unitary operator
\[
W:\, \bigoplus_{i\in I } \hh_\kappa \longrightarrow \hh_{K_w}\,, \quad
W(\oplus_i \varphi_i)(x) = \sum_i \varphi_i(x) \sqrt{\sigma_i} e_i, 
\]
(the operator $W$ is unitary since $\sigma_i>0$ for all $i\in I$, so
that $W$ is injective). 
Equation~(\ref{wtilde}) is finally obtained letting  $\widetilde{W}=W^\ast_w W$. 
\end{proof}
 If in Example~\ref{esempioB}, $\kk$ is  a RKHS of scalar
functions over some set $X^\prime$, then there is a particular choice for the
operator $B$, suggested by Example~\ref{ex_mercer}.
\begin{example}\label{square} 
Let $X$ and $X^\prime$ be two locally compact second countable
topological spaces. Let $\kappa:X\times X\to\cuno$ and $\kappo:X^\prime\times
X^\prime\to\cuno$ be two scalar reproducing kernels on $X$ and $X^\prime$,
respectively. 
\begin{enumerate}
\item[{\rm (i)}] If $I^\prime$ denotes the identity operator on
  $\hh_{\kappo}$, define
\[ K:X\times X\to {\mathcal L}(\hh_{\kappa})\qquad K(x,t)=\kappa(x,t) I^\prime,\]
then $K$ is a $\hh_{\kappo}$-reproducing kernel on
$X$ and the corresponding RKHS $\hh_K$ is
unitarily equivalent to $\hh_{\kappa}\otimes\hh_{\kappo}$ by means of
the feature operator
\[ W: \,\hh_{\kappa}\otimes\hh_{\kappo}\longrightarrow \hh_K\,, \qquad
W(\varphi_1\otimes \varphi_2)(x)=\varphi_1(x) \varphi_2\,.\] 
\item[{\rm (ii)}] Define $\kappa\times \kappo:(X\times X^\prime)\times
  (X\times X^\prime)\to\complex$ as 
\[\left(\kappa\times\kappo\right)(x,x^\prime;t,t^\prime)=\kappa(x,t)\kappo(x^\prime,t^\prime),\]  
then $\kappa\times\kappo$ is a scalar kernel on $X\times X^\prime$ and $\hh_{\kappa\times\kappo}$
is unitarily equivalent to $\hh_K$ by means of the feature operator
\[
\widetilde{W}(f)(x,x^\prime)=\left[f(x)\right](x^\prime)=\scal{f(x)}{\kappo_{x^\prime}}_{\kappo}\qquad
f\in \hh_K.\]
\end{enumerate}   
\end{example}
\begin{proof}
The first part follows from  Example~$\ref{esempioB}$ with
$\kk=\hh_{\kappo}$ and $B=I^\prime$, which is injective.
The second part  is a consequence of 
Proposition~\ref{featuremap} applied to
\begin{eqnarray*}
\gamma:\, X\times X^\prime \longrightarrow {\mathcal L}(\complex; \hh_K) \simeq \hh_K\,,\quad 
(x,x^\prime) \longmapsto W(\kappa_{x}\otimes\kappo_{x^\prime}) \,,
\end{eqnarray*}
taking into account the injectivity of $W$ and the equalities
\begin{eqnarray*}
\scal{W(\varphi_1\otimes\varphi_2)}{\gamma_{(x,x^\prime)}}_K&=&\scal{\varphi_1\otimes\varphi_2}{\kappa_{x}\otimes\kappo_{x^\prime}}=\varphi_1(x)\scal{\varphi_2}{\kappo_{x^\prime}}_{\kappo}\\
&=&\scal{W(\varphi_1\otimes\varphi_2)(x)}{\kappo_{x^\prime}}_{\kappo}=\widetilde{W}(W(\varphi_1\otimes\varphi_2))(x,x^\prime).
\end{eqnarray*}
\end{proof}
By using Proposition~\ref{mercer} on the space $X^\prime$, the above example
can be realized in an alternative way.
\begin{example}\label{square1}
Let $X$ and $X^\prime$ be two locally compact second countable
topological spaces. Let $\kappa:X\times X\to\cuno$ and $\kappo:X^\prime\times
X^\prime\to\cuno$ be two scalar $\Co$-reproducing kernels on $X$ and $X^\prime$,
respectively.  Let $\muu$ be a probability measure on
$X^\prime$ with $\supp{\muu}=X^\prime$ and $L_{\muu}$ be the
integral operator on $L^2(X^\prime,\muu)$ with kernel $\kappo$. Define
\[ \widehat{K}: X\times X\to{\mathcal L}(L^2(X^\prime,\muu))\qquad
\widehat{K}(x,t)=\kappa(x,t) L_{\muu} ,\]
then the kernel $\widehat{K}$ is a $L^2(X^\prime,\muu)$-reproducing
kernel and the space $\hh_{\widehat{K}}$ is
unitarily equivalent to $\hh_{\kappa}\otimes\hh_{\kappo}$ by means of
\[ \widehat{W}(f\otimes g)(x)=f(x) i_{\muu}(g)\qquad
f\in\hh_{\kappa},\ g\in\hh_{\kappo},\]
where $i_{\muu}$ is the inclusion of $\hh_{\kappa^\prime}$ in $L^2(X^\prime,\mu)$.
In particular, $\widehat{K}$ is a $\Co$-kernel.
\end{example}
\begin{proof}
Apply Proposition~$\ref{actionw}$ with $K=\kappa
I^\prime$, as in the previous example, and $w=i_{\muu}$, which is
injective. Clearly $K_w=\widehat{K}$, so that  $\hh_{\widehat{K}}$ is unitarily
equivalent to $\hh_{\kappa I^\prime}$. 
The thesis follows immediately from Example~\ref{square}.
\end{proof}
The above example shows that $\hh_K$ and $\hh_{\widehat{K}}$ are the
same RKHS, where the elements of $\hh_K$ are regarded as functions
from $X$ into 
$\hh_{\kappo}$, whereas the elements of $\hh_{\widehat{K}}$ are
regarded as functions from $X$ into $L^2(X^\prime,\muu)$. 
\subsection{Application to learning theory}
We end this section considering an application of some of
the above examples to vector valued regression problems.
In learning theory, a popular algorithm is the minimization on a RKHS
$\hh_K$ of the empirical error   with a penalty term proportional to
the square of the norm \cite{evpopo00}, namely
\begin{equation}
  \label{RLS}
f^\star=\argmin_{f\in\hh_K} \left(\frac{1}{n}\sum_{\ell=1}^n \no{y^\ell
  -f(x^\ell)}_\kk^2 + \la \no{f}^2_K\right)\,.
\end{equation}
Here $\left\{(x^1,y^1),\dots,(x^n,y^n)\right\}$ is the training set of $n$
input-output pairs $(x^\ell,y^\ell)\in X\times Y$ and $\la>0$ is
the regularization parameter. If the reproducing kernel $K$ is as in
Example~\ref{esempioB}, then it can be checked that
\[ f^\star(x)=\sum_{i\in I}\varphi^\star_i(x) y_i \]
where each $\varphi^\star_i$ is given by
\[
\begin{split}
\varphi^\star_i & =\argmin_{\varphi\in\hh_\kappa}\left(\frac{1}{n}\sum_{\ell=1}^n |y_i^\ell 
  -\varphi(x^\ell)|^2 + \frac{\la}{\sigma_i} \no{\varphi}_\kappa^2\right),  
\end{split}
\] 
and $y_i^\ell=\scal{y^\ell}{y_i}$. \\
In many applications
$\kk=\complex^m$ so that $B$ is a $m\times m$ positive semi-definite
matrix. The above observation reduces the problem of computing the
minimizer of\eqn{RLS} to $|I|$~scalar problems, where the cardinality $|I|$ is
the rank of the matrix $B$.

With the choice of $K$ as in Proposition~\ref{ex_psi}, let $f^\star$ be
the minimizer given by\eqn{RLS}, where the $n$-examples in the
training set are the pairs $(t^\ell,y^\ell)\in T\times \runo^m$. 
By using the fact that $W$ is a partial surjective isometry, one can check that 
\[f^\star(t)=(\varphi^\star(\Psi_i(t)),\ldots, \varphi^\star(\Psi_m(t)),\] 
where $\varphi^\star$ is given by
\[\varphi^\star=\argmin_{\varphi\in\hh_\kappa}\left( \frac{1}{n}\sum_{\ell=1}^n\sum_{i=1}^m |y_i^\ell
  -\varphi(x_i^\ell)|^2 + \la \no{\varphi}_\kappa^2\right)\,,\] 
where $y_i^\ell\in\runo$ are the components of the output $y^\ell\in\runo^m$
and $x_i^\ell=\Psi_i(t^\ell)\in X$. With this choice the
problem\eqn{RLS} is reduced to a minimization problem on the scalar
RKHS $\hh_\kappa$.

\section{Universal kernels: main results}\label{secuniversal}
In this section we address the problem of defining and characterizing
the universality of a kernel 
$K$. As pointed out in the introduction, in learning theory a
necessary condition in order to have universally consistent algorithms is the
assumption that the reproducing kernel Hilbert space $\hh_K$ is dense in $\LduemuXK$ for any 
probability measure $\mu$. From this point of view next definition is very natural.
\begin{definition}\label{Def. di universalita'} 
Let $K:X\times X\to\BK$ be a reproducing kernel.
\begin{enumerate}
\item[{\rm (i)}] A $\Co$-kernel $K$ is called {\bf universal} if $\hh_K$ is dense in
$\LduemuXK$ for each probability measure $\mu$.   
\item[{\rm (ii)}] A Mercer kernel $K$ is called {\bf compact-universal} if
  $\hh_K$ is dense in $\LduemuXK$ for each probability measure $\mu$
  with compact support.  
\end{enumerate}
\end{definition}
We briefly comment on the above definitions. In item~(i)
the assumption that the kernel is $\Co$ ensures both that $\hh_K$ is a subspace of
$\LduemuXK$ and that   universality is equivalent
to the density of $\hh_K$ is $\CoXK$ (see Theorem~\ref{main1}). In item~(ii), since $\mu$ has
compact support, it is enough to assume that $K$ is a Mercer kernel in order to
have  $\hh_K\subset\LduemuXK$. This last property
turns out to be equivalent to the definition of universality given in
\cite{camipoyi08}.\\
Clearly a universal kernel is also compact-universal. 
Conversely, a  $\Co$-kernel can be compact-universal but not
universal, as shown by~Examples~\ref{controesempio} and~\ref{sync}. 

Notice that in Definition~$\ref{Def. di universalita'}$ if  we replace $\LduemuXK$
with $\LpmuXK$ for an arbitrary $1\leq p<\infty$, we have in principle
a different notion of universality. Nevertheless Theorem~$\ref{main1}$ clarifies that
there is no difference. We state the results for $p=2$, since it is
the natural choice in learning theory.

The following corollary shows that universality is preserved by
restriction to a  subset.
\begin{corollary}\label{HZu}
Let $X_0$ be  a subset of $X$.
\begin{enumerate}
\item[{\rm (i)}] If $X_0$ is closed and $K$ is universal, then $K_{X_0}$ is universal.
\item[{\rm (ii)}]  If $X_0$ is locally closed and $K$ is
  compact-universal, then $K_{X_0}$ is compact-universal.
\end{enumerate}
\end{corollary}
\begin{proof}
We only prove~(i). Since $X_0$ is closed, Corollary~\ref{HZ} implies that
$K_{X_0}$ is a $\Co$-kernel, and a function $f$ belongs to $\hh_{K_{X_0}}$ if and only if
there exists $g\in\hh_K$ such that $f=g_{\mid_{X_0}}$. Given a
  probability measure $\mu$ on $X_0$,   let $\nu$  be the probability
  measure on $X$, $\nu(E)=\mu(E\cap X_0)$ for any Borel subset $E$ of
  $X$. By universality of $K$, $\hh_K$ is dense in
  $L^2(X,\nu,\kk)\simeq L^2(X_0,\mu,\kk)$, where the equivalence is
  given by the restriction from $X$ to $X_0$, so that
  $\hh_{K_{X_0}}$ is dense in  $ L^2(X_0,\mu,\kk)$.
\end{proof}
The converse is clearly not true. Notice that the compact-universal
kernels are precisely the Mercer kernels such that $K_{X_0}$ is
universal for any compact subset $X_0$ of $X$. 

\medskip

In the next subsections we discuss separately the two notions of universality and then we 
make a comparison between them. 

\subsection{Universality and $\Co$-kernels}  
In this section we characterize the universal $\Co$-kernels.
First result shows that the density of $\hh_K$ in
$\LduemuXK$ for any probability measure $\mu$ is equivalent to the
density in $\CoXK$ and that one can replace $\LduemuXK$ with $\LpmuXK$, $1\leq p <\infty$.
\begin{theorem}\label{main1}
  Suppose $K$ is a $\Co$-kernel. The following facts are equivalent.
 \begin{itemize}
 \item[{\rm (a)}] The kernel $K$ is universal.
 \item[{\rm (b)}] The space $\hh_K$ is dense in  $\CoXK$.
 \item[{\rm (c)}] There is  $1\leq p <\infty$ such that $\hh_K$ is dense in
   $\LpmuXK$ for all probability measures $\mu$ on $X$.
  \end{itemize}
 \end{theorem}
\begin{proof} 
Clearly (a) implies (c). Since $X$ is locally compact and second
countable,  $\CoXK$ is dense in $\LduemuXK$ where  the inclusion is
continuous,  so that (b) implies (a).  \\
We show that (c) implies (b).
Suppose hence that $\hh_K$ is not dense in $\CoXK$. Then, there
exists $T\in \CoXK^\ast$, $T\neq 0$ such
that $T(f)=0$ for all $f\in \hh_K$. By Theorem \ref{semplice}, there
is a probability measure $\mu$ on $X$ and a function $h\in\LinfmuXK$ such
that $T(f)=\int_{X} \scal{f(x)}{h(x)}\de \mu(x)$. Since $T\neq 0$,
then $h\neq
0$. \\
Since $\mu$ is a probability measure, $h$ is a non-null element in
$L^{p/(p-1)} (X,\mu;\kk)=\LpmuXK^\ast$ (where we set $1/0 = \infty$)
such that
\[ \int_{X} \scal{f(x)}{h(x)}\de \mu(x)=0\qquad \forall f\in\hh_K.\]
It follows that $\hh_K$ is not dense in $\LpmuXK$.
\end{proof}
As a consequence of the previous theorem, we have the following nice corollary.
\begin{corollary}
\label{oneforall}
Let $K$ be a $\Co$- kernel. Given $1\leq p \leq q <\infty$, the space
$\hh_K$ is dense in $\LpmuXK$ for all probability
measures $\mu$ if and only if it is dense in $\LqmuXK$ for all
probability measures $\mu$. 
\end{corollary}
The previous result is not trivial. Clearly, if $q\geq p$, the
space $\LqmuXK$ is always  a dense subspace of $\LpmuXK$ and the
inclusion is continuous. Hence, if a RKHS $\hh_K$  is dense in $\LqmuXK$,
then $\hh_K$ is always dense in $\LpmuXK$. However, in general
$\LpmuXK$ is not contained in $\LqmuXK$, so that,  if $\hh_K$ is dense
$\LpmuXK$, the density of $\hh_K$ in  $\LqmuXK$ has to be
proved. Corollary~\ref{oneforall}  
shows this result under the assumption that $K$ is $\Co$. 

Now, we give a characterisation of universality of $K$ in terms of the
injectivity property of the integral operators $L_\mu$, for $\mu$
varying over the probability measures on $X$. 
\begin{theorem}\label{univimplLKiniett}
Suppose $K$ is a $\Co$-kernel. Then the following facts are equivalent.
\begin{enumerate}
\item[{\rm (a)}] The kernel $K$ is universal.
\item[{\rm (b)}] The operator $i_\mu^\ast:\LduemuXK\to\hh_K$ is an 
injective operator for all probability measures $\mu$ on $X$.
\item[{\rm (c)}] The integral operator $L_\mu:\LduemuXK\to\LduemuXK$
  is injective for all probability measures $\mu$ on $X$. 
\end{enumerate}
\end{theorem}
The proof is an immediate consequence of Theorem \ref{main1} and the
next proposition. 
\begin{proposition}\label{densita-iniettivita}
Let $K$ be a Mercer kernel and $\mu$ a fixed probability measure on
$X$ such that $K$ is bounded on the support of $\mu$. The following
facts are equivalent. 
\begin{enumerate}
\item[{\rm (a)}] The space $\hh_K$ is dense in $\LduemuXK$.
\item[{\rm (b)}] The operator  $i_\mu^\ast$ is injective.
\item[{\rm (c)}] The integral operator $L_\mu$ is injective.
\end{enumerate}
\end{proposition}
\begin{proof}
The space $\hh_K$ is dense in $\LduemuXK$ if and only if the range of
  $i_\mu$ is dense in $\LduemuXK$. This last condition is equivalent
  to the injectivity of $i_\mu^*$, that is,  {\rm (a)} is equivalent to {\rm (b)}.
  Since $L_\mu=i_\mu i_\mu^*$ and $\ker{L_\mu}=\ker{i_\mu^*}$, then {\rm (b)} and
  {\rm (c)} are equivalent.
\end{proof}

\subsection{Compact-universality}\label{42}

In this section, we characterize compact-universality of Mercer
kernels and we show that compact-universality is precisely what
is called universality  in \cite{camipoyi08}.

Next theorem characterizes compact-universality.
 \begin{theorem}\label{main2}
Suppose $K$ is  a Mercer kernel. 
The following facts are equivalent.
\begin{itemize}
\item[{\rm (a)}] The kernel $K$ is compact-universal.
\item[{\rm (b)}] The space $\hh_K$ is dense in $\CXK$ endowed with compact-open topology.
\item[{\rm (c)}]  There is  $1\leq p <\infty$ such that $\hh_K$ is dense in $\LpmuXK$ for all
compactly supported probability measures.
\end{itemize}
\end{theorem}
\begin{proof}
Clearly  (a) implies (c). We prove that (b) implies (a). Indeed, fixed a probability measure
$\mu$ with compact support $Z$, the fact that  $\hh_K$ is dense in
$\CXK$ implies that $\left.\hh_K\right|_Z:=\set{f_{\left.\right|_Z}\mid f\in\hh_K}$ is dense
  in $\CZK$, but $\CZK$ is clearly dense in $\LduemuZK\simeq\LduemuXK$ with
  continuous injection. Hence $\hh_K$
  is dense in $\LduemuXK$. It only remains to prove that (c) implies (b).
 For this,  it is enough to prove that $\left.\hh_K\right|_Z$ is dense in $\CZK$
  with the uniform norm, for all compact subset $Z$ of $X$. But this
  is a simple consequence of Theorem \ref{main1} since
  $\left.\hh_K\right|_Z$ is clearly dense in $\LpmuZK$ for all
  probability measure $\mu$ on $Z$, and $\CZK=\mathcal{C}_0
  (Z;{\kk})$.
\end{proof}

The analog of theorem \ref{univimplLKiniett} also holds.
\begin{theorem}\label{univimplLKiniettme}
Suppose $K$ is a Mercer kernel. Then the following facts are equivalent.
\begin{enumerate}
\item[{\rm (a)}] The kernel $K$ is compact-universal.
\item[{\rm (b)}] The operator $i_\mu^\ast:\hh_K\to\LduemuXK$ is an 
injective operator for all compactly supported probability measures $\mu$ on $X$.
\item[{\rm (c)}] The integral operator $L_\mu:\LduemuXK\to\LduemuXK$ is injective for all
  probability measures $\mu$ on $X$ with compact support. 
\end{enumerate}
\end{theorem}
The proof is a simple consequence of Proposition~\ref{densita-iniettivita}.

Clearly universality of a $\Co$-kernel $K$ implies
compact-universality. The converse is not true as shown by the
following example, see also Example~\ref{sync}.  The reason of this
phenomenon is the fact that 
$\CoXK$ endowed with the compact-open topology is not continuously embedded in $\LpmuXK$.
\begin{example}\label{controesempio}\rm
  Let $X = \Z_+$, and let $\ell^2$ be the Hilbert space of square
  summable sequences. Then, $\ell^2$ is a RKHS of scalar functions on
  $X$ with reproducing kernel $K(i,j) = \delta_{i,j}$, where
  $\delta_{i,j}$ is the Kronecker delta. We fix the following sequence
  $\{ f_k \}_{k\in\Z_+}$ in $\ell^2$
$$
f_k (j) = \delta_{j,k} + e \delta_{j,k+1} ,
$$
and we let
\begin{equation}\label{chiusura in l2}
\hh_{\tilde{K}} = \ell^2 {\rm -cl}\, \Spanno{ f_k \mid k\in\Z_+ }
\end{equation}
($\ell^2 {\rm -cl}$ denotes the closure in
$\ell^2$). $\hh_{\tilde{K}}$ is also a RKHS of scalar functions on
$X$, whose reproducing kernel we denote by $\tilde{K}$. Since $\ell^2
\subset c_0$ ($\, =$ the sequences going to $0$ at infinity),
$\tilde{K}$ is a $\Co$-reproducing kernel.

For all $n\in\Z_+$, let $Z_n = \{ 1,2 \ldots n \}$. $Z_n$ is compact
in $X$, and every compact set $Z\subset X$ is contained in some
$Z_n$. Clearly,
$$
\cc{ Z_n } = \Spanno{(f_k)_{|_{Z_n}} \mid k\leq n}  ,
$$
hence $\hh_{\tilde{K}}$ is dense in $\cccc (X)$ with the topology of
uniform convergence on compact subsets.

Let $\mu$ be the probability measure on $X$ such that $\mu ( \{ j \} )
= (e-1)e^{-j} $. We claim that $\hh_{\tilde{K}}$ is not dense in $L^2
(X,\mu)$. In fact, let $f\in L^2 (X,\mu)$ be the function $f(j) =
(-1)^j$. We have $\scal{f_k}{f}_{L^2 (X,\mu)} = 0$ for all $k$. By
(\ref{chiusura in l2}) and continuity of the inclusion $\ell^2
\hookrightarrow L^2 (X,\mu)$, we see that $f$ is in the orthogonal
complement of $\hh_{\tilde{K}}$ in $L^2 (X,\mu)$. The claim then
follows.
\end{example}
A universal kernel is strictly positive definite, but the
converse in general fails, as shown by the following corollary and example.
\begin{corollary}\label{univimplKstrdefpos}
Suppose $K$ is a compact-universal kernel. Then $K$ is strictly positive
definite, {\em i.e.} for all finite subsets $\{ x_1 , x_2 \ldots x_N \}$ of
$X$ such that $x_i\neq x_j$ if $i\neq j$, the condition 
$$\sum_{i,j = 1}^N \scal{K(x_i , x_j) y_j}{y_i} =
0\quad\quad(y_i\in\kk,\,i=1 \ldots N)$$ 
implies $y_i=0$ for all $i=1,\ldots,N$.
\end{corollary}
\begin{proof}
  Assume $\sum_{i,j = 1}^N \scal{K(x_i , x_j) y_j}{y_i} = 0$ for some
  finite subset $\{ x_1 , x_2 \ldots x_N \} \in X$, $x_i\neq x_j$ if
  $i\neq j$, and  $\{y_1,y_2\ldots x_N\}$ in $\kk$.
  Taking $$\mu=\frac{1}{N}\sum_{i=1}^N\delta_{x_i}\quad\quad\mbox{and}\quad\quad
  \varphi=\sum_{i=1}^Ny_i\delta_{x_i},$$ 
  we obtain a probability measure $\mu$ on $X$ with compact support
  and a function $\varphi\in\LduemuXK$ such that
\begin{eqnarray*}
  0&\,=\,&\sum_{i,j = 1}^N \scal{K(x_i , x_j) y_j}{y_i}=N^2\int_{X\times
    X}\scal{K(x,y)\varphi(y)}{\varphi(x)}\,d\mu(y)\,d\mu(x)\\ 
  &\,=\,&N^2\int_{X}\scal{(L_\mu\varphi)(x)}{\varphi(x)}\,d\mu(x)=
N^2\scal{L_\mu\varphi}{\varphi}_{2}.   
\end{eqnarray*}
Since $L_\mu$ is positive and injective by
Theorem~\ref{univimplLKiniettme}, we have $\varphi(x_i)=0$ for all $i=1,\ldots
,N$. Since $x_i\neq x_j$ if $i\neq j$, then $y_i=0$ for all $i=1,\ldots,N$. 
\end{proof}
The converse of the above corollary fails to be true, as shown by the
following example.
\begin{example}
 Let $K:\runo\times\runo\to\complex$ be the kernel 
\begin{eqnarray*}
K\lft x , t \rgt & = & \int_{-1}^1 e^{2\pi i\lft x -t\rgt p} \de
p= \frac{\sin{2\pi(x-t)}}{\pi(x-t)}. 
\end{eqnarray*}
The map $K$ is a scalar $\Co$-kernel, which is  strictly positive definite,
but not universal.
\end{example}
\begin{proof}
We show that it is strictly positive definite. Indeed, let $x_1 ,
\ldots x_N \in X$ such that $x_i\neq x_j$ if $i\neq j$,
$c_1,\ldots,c_N\in\cuno$ and suppose  
\begin{eqnarray*}
0 =\sum_{i,j=1}^N c_i \overline{c_j} K\lft x_i, x_j\rgt & = & 
\int_{-1}^1 | \sum_{i=1}^N c_i e^{2\pi ix_i p } |^2 \de p 
\end{eqnarray*}
Since $p\mapsto | \sum_i c_i e^{2\pi ix_i p } |^2$ is continuous, it follows that 
$| \sum_i c_i e^{2\pi ix_i p } |^2=0$ for all $p\in [-1,1]$. Observing that
the functions $f_j(t)=e^{2\pi ix_jt}$ are linearly independent on
$[-1,1]$ since $x_i\neq x_j$, it follows that $c_j=0$ for all $j$. Clearly
$K$ 
is a $\Co$-kernel, but it is not universal (see Example \ref{sync}).
\end{proof}

In the next remark we show that compact-universality is exactly what
is called universality in \cite{camipoyi08}. 

\begin{remark}\label{defCaponnetto}\rm 
In \cite{camipoyi08}, a Mercer kernel $K$ is
said to be universal if, for each compact set $Z\subseteq X$ 
\begin{equation}
\label{caponnettouniversale}
\CZK = \unichiusZ{\rm span}\,\left\{ K\left(\cdot, x\right)
  v_{\left.\right|Z} \mid x\in Z , \, v\in \kk\right\} , 
\end{equation}
where $\no{\cdot}_Z {\rm
  -cl}$ denotes the closure in $\CZK$ with the uniform norm topology.
This is equivalent to require that $\hh_K$ is dense $\CXK$ with the
compact-open topology, that is,
by Theorem~\ref{main1} that $K$ is compact-universal. Indeed, by
definition of the compact-open  topology, $\hh_K$ is dense in $\CXK$ if and only
if 
\begin{equation} \label{universaleCZ}
\CZK = \unichiusZ{\left.\hh_K \right|_Z}
\end{equation}
for all compact $Z\subseteq X$. \\
Clearly\eqn{caponnettouniversale} implies
\eqn{universaleCZ}. Suppose on the other hand that
\eqn{universaleCZ} holds true.  Denote with $\widetilde{K}$ the
restriction of $K$ to $Z\times Z$. Since convergence in
$\hh_{\widetilde{K}}$ implies uniform convergence we have
\begin{equation*}
\unichiusZ{\rm span}\left\{K(\cdot,x)v_{\left.\right|_Z}\mid x\in Z,
  \, v\in \kk\right\}  \supseteq  \hh_{\widetilde{K}}
\end{equation*}
On the other hand, $\hh_{\widetilde{K}}= \left.\hh_K\right|_Z$ as a
linear space of functions (see Corollary~\ref{HZ}). Hence
\eqn{universaleCZ} implies \eqn{caponnettouniversale}. 
\end{remark}

\section{Translation invariant kernels and universality}\label{secabelian}

In this section we assume that $X$ is a locally compact second
countable topological group  with identity $e$ and we study the reproducing kernels that
are translation invariant, namely
\begin{equation}
  \label{invariant}
K(zx,zt)=K(x,t) \qquad \text{for all } x,t,z\in X.
\end{equation}
In particular we characterize all the translation invariant kernels in terms of a
unitary representation of $X$ acting on an arbitrary Hilbert space $\hh$ and an
operator $A:\hh\to\kk$. If $X$ is an abelian group, we give a more
explicit characterization in Theorem~\ref{bochteo} and
Theorem~\ref{caldo} provides a 
sufficient condition ensuring 
that the corresponding reproducing kernel Hilbert space is
universal. This condition is also necessary if $X$ is compact or
$\kk=\complex$. For scalar kernels on $\runo^d$ our result has been already
proved in \cite{mixuzh06}.

For a representation $\pi$ of $X$ on a vector space $V$ we mean a
group homomorphism from $X$ to the automorphisms of 
$V$. In particular, if $V$ is a Hilbert space, $\pi$ is unitary if it
takes values in the group of unitary operators on $V$.  In this
framewok the representation is called continuous if $\pi$ is strongly
continuous (see \cite{fol95}).

We denote by $\la$
the left regular representation of $X$ acting on $\kkx$, namely
\[ (\la_x f)(t)=f(x^{-1}t)\qquad t,x\in X,\ f\in\kkx.\]
We recall that a function $\Gamma:X\to\BK$ is {\em of completely positive type} if  
\begin{equation}\label{Gamma def. pos.}
\sum_{i,j = 1}^N \scal{ \Gamma(x_j^{-1}x_i) y_j}{y_i} \geq 0
\end{equation}
for all finite sequences $\{ x_i \}_{i=1 \ldots N}$ in $X$ and
$\{y_i\}_{i=1 \ldots N}$ in $\kk$. 

The following facts are easy to prove.
\begin{proposition}\label{ale}
Let  $K:X\times X\to\BK$ be a reproducing kernel. The following
conditions are equivalent. 
\begin{itemize}
\item[\rm (a)] $K$ is a translation invariant reproducing kernel.
\item[\rm (b)] There is a function $K_e:X\to\BK$  of completely positive type
  such that $K(x,t)=K_e(t^{-1}x)$.
\end{itemize}
If one the above conditions is satisfied, then the representation
  $\la$ leaves invariant $\hh_K$, its action on $\hh_K$ is unitary and
\begin{align}
  \label{aleale}
& K(x,t)=K_e^* \la_{x^{-1}t} K_e\qquad x,t\in X \\
& \no{K(x,x)}=\no{K_e(e)} \qquad \qquad x\in X \label{8a0}
\end{align}
\end{proposition}
The notation $K_e$ for the function of completely positive type associated with
the reproducing kernel $K$ is consistent with the definition 
given by$\eqn{ern}$ since
\[(K_ey)(x)=K_e(x)y  \qquad y\in\kk,\ x\in X.\] 

\begin{proof}[Proof of Proposition~\ref{ale}]
Assume (a). Given $x,t\in X$, \eqn{ern} and\eqn{invariant} give
\[ K_e(t^{-1}x)=K(t^{-1}x,e)=K(x, t).\]
Since $K$ is a reproducing kernel, $K_e$ is of completely positive
type, so that (b) holds true.\\
Assume (b). Clearly $K$ is a translation invariant reproducing kernel,
so that (a) holds true.

Suppose  now that $K$ is a translation invariant reproducing kernel. Observe that, given $t\in
X$ and $y\in\kk$, 
\[ (\la_x K_ty)(z)=(K_t
y)(x^{-1}z)=K(x^{-1}z,t)y=K(z,xt)y=(K_{xt}y)(z)\qquad x,z\in X,\] 
that is, $\la_xK_t=K_{xt}$. Moreover
\[
\begin{split}
\scal{\la_xK_{t_1}y_1}{\la_xK_{t_2}y_2}_K& =\scal{K_{xt_1}y_1}{K_{xt_2}y_2}_K=
\scal{K(xt_2,xt_1)y_1}{y_2} \\
&=\scal{K(t_2,t_1)y_1}{y_2}=\scal{K_{t_1}y_1}{K_{t_2}y_2}_K.  
\end{split}
\]
This means that $\la$ leaves  
the set $\set{K_xy\mid x\in X, y\in\kk}$ invariant and its action is
  unitary. First two claims now follow recalling that $\set{K_xy\mid x\in X, y\in\kk}$ 
is total in $\hh_K$. To prove\eqn{aleale} observe that
\[K(x,t)=K_x^*K_t= K_e^*\la_x^*\la_t K_e=K_e^*\la_{x^{-1}t} K_e\]
for all $x,t\in X$. 
\end{proof}
Notice that, if $K$ is a translation invariant kernel,\eqn{8a0}
implies that the elements of $\hh_K$ are bounded functions.
The following lemma characterizes the translation invariant kernels
that are Mercer or $\Co$.  
\begin{lemma}\label{abelianco}
Let $K_e:X\to\kk$ be a function of completely positive type and let $K$ be
the corresponding translation invariant reproducing kernel. The
following conditions are equivalent. 
\begin{itemize}
\item[\rm (a)] The map $K$ is a Mercer kernel.
\item[\rm (b)] For all $y\in\kk$, $K_e(\cdot)y\in\CXK$. 
\item[\rm (c)] The representation $\la$ is continuous on $\hh_K$.
\end{itemize}
Moreover, the map $K$ is a $\Co$-kernel if and only if $K_e(\cdot)y\in\CoXK$ for
  all $y\in\kk$.
\end{lemma}
\begin{proof}
The equivalence between (a) and (b) as well as the statement about
$\Co$-kernel is a consequence of Proposition~\ref{equivalenza in Co},
observing that $(K_xy)(t)=K_e(x^{-1}t)y$ and\eqn{8a0} holds. \\ 
Assume that $K$ is a Mercer kernel. Since $\la$ is a unitary
representation and the set $\{K_ty\mid t\in X,y\in\kk\}$ is total in
$\hh_K$, it is enough to  
check that for any $t\in X$ and $y\in\kk$ the function $x\mapsto \la_xK_ty$
is continuous at the identity. Indeed, observe that
\[
\begin{split}\no{\la_x K_t y - K_t y}_K^2& =\no{K_{xt}y-K_t y}_K^2 \\
& = \scal{\left(K(xt,xt)-K(t,xt)-K(xt,t)+K(t,t)\right)y}{y} \\
& = \scal{\left(2K_e(e)-K_e(t^{-1}x^{-1}t)-K_e(t^{-1}xt)\right)y}{y} 
\end{split},\]
which is continuous at the identity by assumption on
$K_e$. Conversely, if $\la$ is continuous,\eqn{aleale} gives that
\[ K_e(x)y=K(x,e)y=K_e^*\la_{x^{-1}}K_ey,\]
so that $K_e(\cdot)y$ is continuous.
\end{proof}
The following theorem characterizes the translation invariant
reproducing kernels.
\begin{proposition}\label{erni}
Let $\pi$ be a unitary representation of $X$ acting on a separable Hilbert space
$\hh$ and $A:\hh\to\kk$  a bounded operator. Define  
\beeq{w}{ W:\hh\to\kkx\,, \qquad (W v)(x)=A\pi_{x^{-1}} v\,.}
$W$ is a unitary map from $\ker{W}^\perp$ onto the reproducing
kernel Hilbert space $\hh_K$ with translation invariant kernel
\begin{equation}
  \label{ernern}
K(x,t)= A\pi_{x^{-1}t} A^* \qquad x,t\in X.  
\end{equation}
Moreover $W$ intertwines the representations $\pi$ and $\la$. Finally
$W$ is unitary if and only if the only $\pi$-invariant closed subspace
of $\ker{A}$ is the null space.
\end{proposition}
\begin{proof}
Define $\ga_x:\kk\to\hh$ as $\ga_x=\pi_xA^*$, so that
$(Wv)(x)=\ga_x^*v=A\pi_{x^{-1}}v$. The claim is now consequence of
Proposition~\ref{featuremap}, up the last statement. The fact that $W$
intertwines $\pi$ with $\la$ is trivial. Finally, by
Proposition~\ref{featuremap}, $W$ is unitary if and only if is
injective. By definition
\[\ker{W}=\set{v\in\hh\mid \pi_xv\in\ker{A}\ \forall x\in X}.\]
Hence $\ker{W}$ is a closed subspace of $\ker{A}$ invariant with respect to $\pi$.  
Conversely any $\pi$-invariant closed subspace of $\ker{A}$ is contained in $\ker{W}$.
\end{proof}
Proposition~\ref{ale} and~\ref{erni} show that any translation
invariant kernel is of the form $K(x,t)= A\pi_{x^{-1}t} A^*$ for some unitary representation
$\pi$ acting on a Hilbert space $\hh$ and a bounded operator
$A:\hh\to\kk$.  In particular,  if $\pi$ is a continuous representation, then
$K$ is a Mercer kernel and for any Mercer kernel $\pi$ can be assumed
to be continuous and $\hh$ separable. Moreover, the reproducing kernel
Hilbert space 
$\hh_K$ is embedded in $\hh$ by the feature operator $W$ defined by\eqn{w}. 
Observe that if the representation $\pi$ is irreducible or if $A$ is
injective, then $W$ is unitary.  \\
If $\kk=\complex$, the operator $A$ is of the form $Av=\scal{v}{w}_\hh$
for some $w\in\hh$, so that $(Wv)(x)=\scal{v}{\pi_xw}_\hh$. This operator
is well know in harmonic analysis as {\em wavelet operator} \cite{fuhr05}.
\begin{remark}\rm
Notice that any translation invariant kernel $K$ is the sum of translation invariant
kernels associated with cyclic representations. Indeed, let $\pi$ be
a unitary representation defining $K$ by means of\eqn{ernern}. 
Since any unitary representation is the direct sum of a family of cyclic representations,
then $\hh=\oplus_{i\in I} \hh_i$ where each $\hh_i$ is a closed
$\pi$-invariant subspace and the action of $\pi$ on $\hh_i$ is
cyclic. Denote by
$P_i$ the orthogonal projection on $\hh_i$, then 
\[K(x,t)=\sum_{i\in I} AP_i\pi_{x^{-1}t}P_iA^* =\sum_{i\in
  I}K^i(x,t),\] where the series converges in the strong operator
topology and the reproducing kernels $K^i$ are
$K^i(x,t)=A_i\pi^i_{x^{-1}t}A_i^*$ where $\pi^i$ and $A_i$ are the
restrictions of $\pi$ and $A$ to $\hh_i$,
respectively. Proposition~\ref{suma} implies that
$\hh_K=\sum_{i\in  I}\hh_{K^i}$. \\
For scalar kernels, we can always assume that $\pi$ is cyclic
itself. Indeed, the wavelet operator is $(Wv)(x)=\scal{v}{\pi_xw}_\hh$ for
some $w\in \hh$, so that the associated kernel $K$ is determined only
by the cyclic subrepresentation of $\pi$ containing $w$.
\end{remark}

\subsection{Abelian groups}\label{sec_abeliano}

In this section, we specialize the previous discussion to the case in
which $X$ is an abelian group. With this assumption, we can give a
more explicit construction of translation invariant Mercer kernels,
which is related to a generalization of Bochner theorem for scalar
functions of positive type, \cite{boc59,faha72}.

We denote the product in $X$ additively and the identity by $0$,
since the main example is $\runo^d$.  We let $\hat{X}$ be the {\em
dual group} of $X$ and we denote by $\de x$ the Haar measure on $X$.

Now, we briefly recall the definition of Fourier
transform, see for example \cite{fol95}. 
If $\phi\in \LunolXK$, its Fourier transform
$\F{\phi}:\hat{X}\to \kk$ is given by
\[\F{\phi}(\chi)=\int_X \overline{\chi(x)}\ \phi(x)\de x . \]
We denote by $\de \chi$ the Haar measure on $\xh$ normalized so that
$\ff$ extends to a unitary operator from $L^2 (X,\de x ; \kk)$ onto
$L^2 (\xh,\de \chi ; \kk)$. 
If $\mu$ is a positive measure on $X$  and $\varphi\in L^1(X,\mu;\kk)$,
let $\F{\varphi\mu}:\hat{X}\to \kk$ be given by 
\[\F{\varphi\mu}(\chi)=\int_X \overline{\chi(x)}\varphi(x)\ \de\mu(x) .\]
If $\mu$ is a complex measure\footnote{That is, a $\sigma$-additive map
  $\mu:{\mathcal B}(X)\to\complex$.} on $X$, we denote $\F{\mu}=\F{h|\mu|}$
where $|\mu|$ is the total variation of $\mu$ and $h\in L^1(X,|\mu|)$
is the density of $\mu$ with respect to $|\mu|$.

By general properties of Fourier transform, $\F{\phi}$ and
$\F{\mu}$ are bounded continuous functions on $\xh$ (actually,
$\F{\phi} \in\CoXK$). Moreover, $\F{\phi} = 0$ [respectively,
$\F{\mu} = 0$] if and only if $\phi = 0$ in $\LunolXK$ [resp.,
$\mu = 0$].

We recall that a {\em positive operator valued measure} ({\em POVM}) on $\xh$
with values in $\kk$ is a map $Q : \bo{\xh} \frecc \elle{\kk}$ such
that $Q(\zh) \geq 0$ for all $\zh\in\bo{\xh}$, and 
\[
\sum_i Q(\zh_i) = Q(\cup_i \zh_i),
\]
for every denumerable sequence of disjoint Borel sets $\{
\zh_i \}_i$ where the sum converges in the weak  operator topology. A positive
operator valued measure $Q$ is a {\em projection valued  measure} if $Q(\zh)^2
= 1$ for all $\zh\in\bo{\xh}$. If $\fh:\xh\to\complex$ is a bounded measurable
function, $\int_{\xh}\fh(\chi) \de Q(\chi)$ is the unique bounded
operator $\fh(Q)$ defined by
\[ \scal{\fh(Q)y}{y^\prime}=\int_{\xh} \fh(\chi) \de
Q_{y,y^\prime}(\chi)\qquad y,y^\prime\in\kk,\]
where $Q_{y,y^\prime}$ is the complex measure on $\xh$ given by
$Q_{y,y^\prime}(\zh)=\scal{Q(\zh)y}{y^\prime}$ for all Borel subsets $\zh$.

Next theorem shows that there is a one to one correspondence between
translation invariant Mercer kernels on $X$ and positive operator valued
measures on $\xh$. For scalar kernels this result is Bochner theorem
\cite{boc59}. For vector valued kernels, it is proved in
\cite{falb69,faha72} under the 
weaker assumption that $K_0$ is a function of positive type, namely
that 
\begin{equation}\label{posB}
\sum_{i,j = 1}^N c_i\overline{c_j}\scal{ K_0 (x_i - x_j) y}{y} \geq 0
\end{equation}
for all finite sequences $\{ x_i \}_{i=1 \ldots N}$ in $X$,
$\{c_i\}_{i=1 \ldots N}$ in $\complex$ and $y\in\kk$. The fact that
conditions\eqn{Gamma def. pos.} and\eqn{posB} are equivalent for
abelian groups is a consequence of \cite[Lemma 3.1]{davies}.
In the following, assuming\eqn{Gamma def. pos.}, we give a 
proof simpler than the one provided in \cite{falb69,faha72}. 
\begin{theorem}\label{bochteo}
If $Q : \bo{\xh} \frecc \elle{\kk}$ is a positive operator valued measure, then
\begin{equation}\label{bochint}
K(x,t) = \int_{\xh} \chi(t-x) \de Q (\chi)
\end{equation}
is a translation invariant $\kk$-Mercer kernel on $X$. Conversely, if
$K$ is a translation invariant $\kk$-Mercer kernel on $X$, then there
exists a unique positive operator valued measure $Q$ such that
$(\ref{bochint})$ holds. 
\end{theorem}
We say that $Q$ in (\ref{bochint}) is the  positive operator
  valued measure {\em associated} to the translation invariant Mercer kernel
$K$.\\ 
\begin{proof}[Proof of Theorem~\ref{bochteo}]  
If $Q:\bo{\xh} \frecc
  \elle{\kk}$ is a positive operator valued measure, by Neumark
  dilation theorem \cite{rina90} there exist a separable Hilbert space
  $\hh$, a projection valued measure $P:\bo{\xh} \frecc \elle{\hh}$
  and a bounded operator $A: \hh \frecc \kk$ such that 
\beeq{Naimark}{ Q(\zh) = A P(\zh) A^\ast \qquad \forall \zh\in\bo{\xh} .  }
Let $\pi$ be the continuous unitary representation of $X$ acting on
$\hh$ given by
\beeq{SNAG}{ \pi(x) = \int_{\xh} \chi(x) \de P(\chi) , } 
see~\cite{fol95}. Eq.\eqn{bochint}
  then becomes $K(x,t) = A\pi_{ t-x} A^\ast$, so that $K$ is a
  translation invariant Mercer kernel by Proposition~\ref{erni} and Lemma~\ref{abelianco}. \\
Conversely, by Proposition~\ref{erni} and Lemma~\ref{abelianco}, every
translation invariant Mercer kernel is of 
  the form $K(x,t) = A\pi_{ t-x} A^\ast$ for some continuous
  unitary representation $\pi$ of $X$ in a separable Hilbert space
  $\hh$ and some bounded operator $A:\hh\frecc\kk$. By SNAG theorem
  \cite{fol95}, there is then a projection valued measure
  $P:\bo{\xh}\frecc\elle{\hh}$ such that\eqn{SNAG} holds and\eqn{bochint}
 follows defining the POVM $Q$ as in\eqn{Naimark}. \\
Finally, uniqueness of $Q$ follows from
\[
\scal{K_0 (x) y}{y^\prime} = \int_{\xh} \overline{\chi(x)} \de Q_{y,y^\prime} (\chi)
=\F{Q_{y,y^\prime}}(x)\]
by injectivity of Fourier transform of measures on $\xh$.
\end{proof}

The next proposition is a useful tool to construct translation invariant Mercer kernels.
\begin{theorem}\label{abelianern}
Let $\nuh$ be a measure on $\xh$ and $A:\LdueXE\to \kk$ be
a bounded operator. For all $y,y^\prime\in\kk$ let
\beeq{ff}{\scal{K(x,t)y}{y^\prime}=\int_{\xh} \chi(t-x)
  \scal{(A^*y)(\chi)}{(A^*y^\prime)(\chi)} \ 
\de\nuh(\chi).}
Then $K$ is a translation invariant Mercer kernel and the
corresponding reproducing kernel Hilbert space is embedded in
$\LdueXE$ by means of the feature operator $W:\LdueXE\to \hh_K$
\begin{align}
  \label{fff}
  (W\fh)(x) & = A \fh^x \qquad\text{where }\qquad
  \fh^x(\chi)=\overline{\chi(x)}\fh(\chi)\\ 
 \scal{(W\fh)(x)}{y} & =\int_{\xh}\overline{\chi(x)}
 \scal{\fh(\chi)}{(A^*y)(\chi)}\ \de\nuh(\chi) \nonumber.
\end{align}
Conversely, any translation invariant Mercer kernel is of the above form for
some positive measure $\nuh$ and bounded operator $A:\LdueXE\to\kk$.
\end{theorem}
\begin{proof}
If $\nuh$ is a measure on $\xh$ and $A:\LdueXE\to \kk$ is a bounded
operator, then 
\[
\scal{Q(\zh) y}{y^\prime} = \int_{\zh} \scal{\lft A^\ast y \rgt (\chi)}{\lft
  A^\ast y^\prime \rgt (\chi)} \de \nuh (\chi) \qquad \forall
\zh\in\bo{\xh},\ y , y^\prime \in \kk 
\]
defines a positive operator valued measure $Q: \bo{\xh}\frecc
\BK$, since $Q(\zh)=AP(\zh)A^\ast$ where $P(\zh)$ is the
multiplication by the characteristic function of $\zh$.
The kernel $K$ given in (\ref{ff}) is then the translation
invariant Mercer kernel associated to $Q$ by (\ref{bochint}). 
To prove\eqn{fff}, set
\[
\gamma_x : \kk \frecc \LdueXE \qquad \lft \gamma_x y \rgt (\chi) =
\chi (x) (A^\ast y)(\chi)  ,
\] 
so that $K(x,t) = \gamma_x^\ast \gamma_t$ and
\begin{align*}
\scal{\gamma_x^\ast \fh}{y} & =  \scal{\fh}{\gamma_x y}_2 = 
\int_{\xh} \scal{\fh (\chi)}{\chi (x) (A^\ast y)(\chi)}\de\nuh(\chi) \\
& =  \int_{\xh} \overline{\chi (x)} \scal{\fh (\chi)}{(A^\ast
  y)(\chi)}\de\nuh(\chi) = \scal{A\fh^x}{y}
\end{align*}
for all $\fh\in\LdueXE$.\\
Conversely, assume that $K$ is a translation invariant Mercer kernel. 
We first consider the case that $\kk$ is
infinite-dimensional. Propositions~\ref{ale} 
and~\ref{erni} show that $K$ is of the form $K(x,t)= A\pi_{t-x} A^*$
for some unitary continuous representati§<on $\pi$ acting on a
separable Hilbert
space $\hh$ and a bounded operator $A:\hh\to\kk$. \\
A basic result of commutative harmonic analysis (see \cite{fol95})
ensures that, for each $n\in \nat_\ast:=\nat\cup\{\infty\}$, there exist a
  complex separable Hilbert space $\kk_n$ of dimension $n$, and a
  measurable subset $\hat{X}_n$ of $\xh$ endowed with a positive measure
  $\nuh_n$ such that the $\hat{X}_n$ are disjoint and cover
  $\hat{X}$. Without loss of generality, we can assume that
  $\nuh_n(\xh_n)\leq 2^{-n}$ and $\nuh_{\infty}(\xh_\infty)\leq
  1$. Moreover there exists a unitary operator 
$U:\, \hh  \to \bigoplus_n L^2(\hat{X}_n, \nuh_n, \kk_n) $ 
such that
\begin{align*}
(U\pi_x U^*\fh_n)(\chi) & =\chi(x)\fh_n(\chi)\qquad \fh_n \in  L^2(\hat{X}_n, \nuh_n, \kk_n)\,.
\end{align*}
For each $n\in \nat_\ast$, let $J_n:\E_n\to \E$ be a fixed isometry, which
always exists since $\E$ is infinite dimensional, and consider the
Hilbert space $\LdueXE$, where $\nuh=\sum_{n}\nuh_n$, which is a
bounded measure by assumption on $\nuh_n$.  Define the isometry 
$V:\, \hh  \to \LdueXE$
as  
\[(Vu)(\chi)= J_n (Uv)(\chi)\qquad \chi\in \xh_n.\]
A simple calculation shows that 
\[ \pi_x =V^*\lah_x V\]
where $(\lah_x\fh)(\chi)=\chi(x)\fh(\chi)$ is the diagonal
representation on $\LdueXE$. Now
\[ K(x,t)= A\pi_{t-x} A^*=AV^*\lah_{t-x}VA^*.\]
Redefining $A=AV^*$,\eqn{ff} is a consequence of
the explicit form of $\lah_x$. 

If $\kk$ is finite dimensional, let $(\nuh,B)$ be the pair associated to
$K$ as in Proposition~\ref{caldo} below. Eq.\eqn{ff} follows defining
$A:L^2(\xh,\nuh,\kk)\to \kk$ 
\[ \scal{A\fh}{y}=\int_{\xh}\scal{B(\chi)^{\frac12}\fh(\chi)}{y}\ \de\nuh(\chi).\]
\end{proof}
If $\kk=\complex^m$, $K(x,t)$ can be regarded as a $m\times m$-matrix and 
$A$ is uniquely defined by a family of functions
$\fh_1,\ldots,\fh_m\in\LdueXE$ through $A^\ast
e_i=\fh_i$. Hence,\eqn{ff} becomes
\begin{equation}
  \label{spagna}
 K(t-x)_{ij}  = \int_{\xh} \chi(t-x) \scal{f_j(\chi)}{f_i(\chi)}\,
\de\nuh(\chi)\qquad i,j=1,\ldots,m.
\end{equation}
As an application,  we give the following  example that generalizes
the one given in \cite{camipoyi08}. 
\begin{example}
Let $X=\runo^d$, regarded as vector abelian group, and
$\kk=\complex^m$. The dual group is isomorphic to $\runo^d$ by means
of $\chi_p(x)=e^{i2\pi x\cdot p}$. Let $\nuh=\de p$ be the Lebesgue
measure on $\real^d$ and
\[\fh_i(p)= \frac{1}{(2\pi)^{d/4}}\,e^{-\sigma_i^2\frac{| p |^2}{2}}\
v_i\qquad v_i\in\kk,\ 
  \sigma_i>0,\] 
then the translation invariant Mercer kernel given by$\eqn{spagna}$
is  
\[\begin{split}
K(t-x)_{ij} & = \int_{\runo^d} e^{i2\pi (t-x)\cdot p} \scal{f_j(p)}{f_i(p)}\,
\de p \\
& = \frac{1}{(\sigma_i^2+\sigma_j^2)^{d/2}}\,
e^{-2\pi^2\frac{|x-t|^2}{\sigma_i^2+\sigma_j^2}} \scal{v_j}{v_i}. 
\end{split}\]
The example in $\cite{camipoyi08}$ corresponds to the choice $v_i=v_j$ and
$\sigma_i=\sigma_j$ for any $i,j=1,\ldots,m$.
\end{example}
Theorems~\ref{bochteo} and~\ref{abelianern} give two
different characterizations of a translation invariant kernel $K$, but
the POVM $Q$ defining $K$ through\eqn{bochint} is always unique,
whereas there are many pairs $(\nuh,A)$ defining the same $K$
by\eqn{ff}. These two descriptions are related observing that, given a
pair $(\nuh,A)$, the scalar bounded measure $Q_{y,y^\prime}$ has
density $ \scal{(A^{*}y)(\chi)}{(A^{*}y^\prime)(\chi)}$ with respect
to $\nuh$ for any $y,y^\prime\in\kk$. On the other hand, given the
POVM $Q$, let $\nuh_Q$ be the bounded positive measure defined by
\beeq{nuq}{
\nuh_Q(\zh) = \sum_n 2^{-n} \no{y_n}^{-2n} \scal{Q(\zh) y_n}{y_n}
\qquad \forall \zh \in\bo{\xh} }
where $\{y_n\}_{n\in \nat}$ is a dense sequence in
$\kk$. Clearly, given
$\zh\in\bo{\xh}$, $\nuh_Q(\zh) = 0$ if and only if $Q(\zh) = 0$,  and
$\nuh_Q$ is uniquely defined by $Q$ up to an  
equivalence. Moreover, by Neumark dilation theorem, see\eqn{Naimark},
there exists an operator $A_Q: 
L^2(\xh,\nuh_Q;\kk)\to\kk$ such that the pair $(\nuh_Q,A_Q)$ gives the
kernel $K$ associated with $Q$.

We notice that in general it is not true  that the POVM $Q$ has an
operator valued density. We recall that $Q$ {\em has operator density}
if there exists a map $B: \xh \frecc \elle{\kk}$  and a positive
measure $\nuh$ such that
$\scal{B(\cdot) y}{y^\prime} \in L^1 (\xh,\nuh)$ for all
$y,y^\prime\in\kk$ and 
\beeq{insalata}
{\int_{\zh} \scal{B(\chi) y}{y^\prime} \de \nuh (\chi) =
Q_{y,y^\prime}(\zh) \qquad \forall \zh\in\bo{\xh}. } 
The following proposition will characterize the kernels having a POVM
with an operator density. To prove the result, we need the following
technical lemma.
\begin{lemma}\label{Q(Z)inL(Y)}
Let $\nuh$ be a positive measure on $\xh$ and  $B: \xh \frecc
\elle{\kk}$ such that 
$\scal{B(\cdot) y}{y^\prime} \in L^1 (\xh,\nuh)$ for all
$y,y^\prime\in\kk$. Then, the sesquilinear form
\begin{equation}\label{formaQ}
\kk\times\kk\to L^1 (\xh,\nuh),\qquad (y,y^\prime)\mapsto
\scal{B(\cdot) y}{y^\prime} 
\end{equation}
is continuous.
\end{lemma}
\begin{proof}
For fixed $y\in\kk$ [resp.~$y^\prime\in\kk$] the map $y^\prime\mapsto
\scal{B(\cdot) y}{y^\prime}$ [resp.~$y\mapsto \scal{B(\cdot)
  y}{y^\prime}$] is continuous from $\kk$ into $L^1 (\xh,\nuh)$ by the
closed graph theorem, {\em i.e.} the application defined in
(\ref{formaQ}) is separately continuous in $y$ and 
$y^\prime$.  So, the closed graph theorem again assures the joint continuity.
\end{proof}
\begin{proposition}\label{caldo}
Let $\nuh$ be a positive measure on $\xh$ and  $B: \xh \frecc \elle{\kk}$ such that
$\scal{B(\cdot) y}{y^\prime} \in L^1 (\xh,\nuh)$ for all
$y,y^\prime\in\kk$ and  $B(\chi) \geq 0$ for $\nuh$-almost all $\chi$.
Then
\beeq{bochnervec}{
K(x,t) =  \int_{\xh} \chi(t-x) B(\chi)\ \de\nuh(\chi) ,
}
is a translation invariant Mercer kernel, and the space 
$\hh_K$ is embedded in $\LdueXE$ by means of the
feature operator
\beeq{bochfea1}{
(W\fh)(x)=\int_{\xh} \overline{\chi(x)} B(\chi)^{\frac 12}\fh(\chi) \de\nuh(\chi) ,
}
where both the above integrals converge in the weak sense.\\
If $\kk$ is finite dimensional or $X$ is compact, any translation
invariant kernel is of the above form for some pair $(\nuh,B)$.\\
If $\kk=\complex$, one can always assume that  $B=1$ and $\nuh$ is a
bounded positive measure.
\end{proposition}
\begin{proof}
Let $\nuh$ and $B$ as in the assumptions. Given a Borel subset $\zh$
of $\xh$ define $Q(\zh)$ as the unique bounded operator satisfying 
\[ \scal{Q(\zh)y}{y^\prime}=\int_{\zh} \scal{B(\chi)y}{y^\prime}\de\nuh(\chi).\] 
The fact that $Q(\zh)$ is a bounded operator follows from
Lemma~\ref{Q(Z)inL(Y)} and from the continuity of the map $L^1 (\xh,\nuh)
\ni \phi \mapsto \int_{\zh}\phi(\chi)\de\nuh(\chi) \in
\cuno$. Clearly, $Q(\zh)$ is a positive operator and monotone 
convergence theorem implies that $\zh\mapsto Q(\zh)$ is a POVM on
$\xh$. By construction ${K(x,t)=\int_{\xh} \chi(t-x) \de
Q(\chi)}$, so  $K$ is a translation invariant Mercer kernel by
Theorem~\ref{bochteo}.   Setting
\[
\gamma_x : \kk \frecc \LdueXE \qquad \lft \gamma_x y \rgt (\chi) =
\chi (x) B(\chi)^{1/2} y , 
\]
we see that $K(x,t) = \gamma_x^\ast \gamma_t$ and
\begin{align*}
\scal{\gamma_x^\ast \fh}{y} & =  \scal{\fh}{\gamma_x y}_2 = 
\int_{\xh} \scal{\fh (\chi)}{\chi (x) B(\chi)^{1/2} y}\de\nuh(\chi) \\
& =  \int_{\xh} \overline{\chi (x)} \scal{B(\chi)^{1/2} \fh
  (\chi)}{y}\de\nuh(\chi) 
\end{align*}
for all $\fh\in\LdueXE$, from which (\ref{bochfea1}) follows.\\
Assume now that $\kk$ is finite dimensional or $X$ is compact and $K$
is a translation invariant Mercer kernel. Theorem~\ref{bochteo}
ensures that there
exists a POVM $Q$ on $\xh$ taking value in $\kk$ such that
${K(x,t)=\int_{\xh} \chi(t-x) \de 
Q(\chi)}$. If $X$ is compact, $\xh$ is discrete. Let $\nuh$ be the
counting measure and $B(\chi)=Q(\set{\chi})$ for all $\chi\in\xh$,
then $(\nuh,B)$ satisfies the required properties.\\
If $\kk$ is finite dimensional, choose $\nuh_Q$ as in\eqn{nuq}. 
It follows that for any $y,y^\prime\in\kk$, the complex measure
$Q_{y,y^\prime}$ has density $b_{y,y^\prime}\in L^1(\xh,\nuh_Q)$ with
respect to $\nuh_Q$. In particular, $b_{y,y}(\chi)\geq 0$ for $\nuh_Q$-almost all
$\chi\in\xh$. Let $y_1,\ldots,y_N$ be a basis of
$\kk$ and by linearity extend $b_{y_i,y_j}\in L^1(\xh,\nuh_Q)$ to a map 
$B:\xh\to\BK$, which clearly satisfies the required properties.\\
If $\kk=\complex$, the claim is clear.
\end{proof}
If $\kk=\complex$, Proposition~\ref{caldo} is already given
in~\cite{mixuzh06}.

We end by showing a sufficient condition ensuring that
  a translation invariant Mercer kernel is of the form given in Proposition~\ref{caldo}.
\begin{proposition}\label{prop. K in L1}
Let $K$ be a translation invariant Mercer kernel. Suppose that
$\scal{K_0 (\cdot) y}{y^\prime} \in L^1 (X,\de x)$ for all $y,y^\prime
\in \kk$. Let 
\beeq{B se K e' L1}{
\scal{B(\chi) y}{y^\prime} := \int_X \chi(x) \scal{K_0 (x) y}{y^\prime}
\de x \qquad \forall y,y^\prime \in \kk .
}  
Then
\begin{itemize}
\item[{\rm (i)}] $B(\chi)$ is a bounded nonnegative operator for all $\chi\in\xh$;
\item[{\rm (ii)}] $\scal{B(\cdot) y}{y^\prime} \in L^1 (X,\de x)$ for all $y,y^\prime \in\kk$;
\item[{\rm (iii)}] for all $x,t\in X$,
\beeq{altroB}{
K(x,t) = \int_{\xh}\chi(t-x) B(\chi) \de \chi ,
}
where the integral converges in the weak sense.
\end{itemize}
\end{proposition}
\begin{proof}
The operator $B(\chi)$ defined in (\ref{B se K e' L1}) is 
bounded as a consequence of Lemma~\ref{Q(Z)inL(Y)} (applied to $K_0$)
and of the continuity of the map $L^1 (X,\de x) 
\ni \phi \mapsto \F{\phi}(\chi) \in \cuno$. \\
Since $\scal{K_0 (\cdot) y}{y}$ is a function of positive type, by
Fourier inversion theorem $\scal{B(\cdot) y}{y} \in L^1 (\xh,\de
\chi)$, and 
\[
\scal{K_0 (x) y}{y} = \int_{\xh} \overline{\chi(x)} \scal{B(\chi) y}{y} \de \chi ,
\]
which is (\ref{altroB}).
\end{proof}

\subsection{Universality}
In this section we study the universality problem for
translation invariant kernels on an abelian group in terms of the
characterization given by Theorem~\ref{bochteo} and
Proposition~\ref{caldo}. The assumptions and
notations are as in Section~\ref{sec_abeliano}.  
To state the following result, we recall that
the support of a POVM $Q$ is the complement of the largest
open subset $U$ such that $Q(U)=0$. 
\begin{proposition}\label{necessita}
Let $K$ be a translation invariant Mercer kernel, and $Q$ its
associated positive operator valued measure. 
If the RKHS  $\hh_K$  is dense in $\LduemuXK$ for any probability
measure $\mu$, then $\operatorname{supp}(Q)= \xh$. 
\end{proposition}
\begin{proof}
Suppose there is an open set $U\subset\xh$ such that $Q (U) =
0$. Let $\chi_0 \in U$, so that $\chi_0 U^{-1}$ is a
neighborhood of the identity element of $\xh$. Let $\mu$ be a
probability measure\footnote{For example, if $V$ is a compact symmetric
  neighborood of the identity of $\xh$ such that $V^2 \subset \chi_0
  U^{-1}$, let $h = 
  1_V \ast 1_V$, so that (up to a constant) the measure $\de \mu (x) =
  \ff^{-1} (h) (x) \de x = \left| \ff^{-1} (1_V) (x) \right|^2 \de x$
  has the required property.}  on $X$ such that ${\rm supp}\,\F{\mu} \subset
\chi_0 U^{-1}$ and  set $\varphi (x) = \overline{\chi_0
  (x)}y$ with $y\in\kk\setminus \{ 0 \}$. Then~(\ref{bochint}) gives 
\begin{eqnarray*}
\scal{L_\mu \varphi}{\varphi} & = &
\int_{X}\int_{X}\int_{\xh} \chi(t-x) \chi_0(x)   \overline{\chi_0(t)}  \de Q_{y,y}
(\chi) \de \mu (x) \de \mu (t) \\ 
& = & \int_{\xh} \left| \F{\mu}(\chi_0\chi^{-1})\right|^2 \de Q_{y,y}(\chi)=0.
\end{eqnarray*}
This shows that $L_\mu$ is not injective, {\em 
  i.e.}~$K$ is not universal. 
\end{proof}
We now characterize the universality of the kernels defined in terms
of the pair $(\nuh,B)$ by means of\eqn{bochnervec}.
\begin{proposition}\label{suff2}
Given a positive measure $\nuh$ on $\xh$ and  $B: \xh \frecc
\elle{\kk}$ such that 
$\scal{B(\cdot) y}{y^\prime} \in L^1 (\xh,\nuh)$ for all
$y,y^\prime\in\kk$ and  $B(\chi) \geq 0$ for $\nuh$-almost all $\chi$,
let $K$ be the translation invariant Mercer kernel given by$\eqn{bochnervec}$.
\begin{itemize}
\item[{\rm (i)}] If $\hh_K$  is dense in $\LduemuXK$ for any
probability measure $\mu$, then both $\supp{\nuh}=\xh$ and $\supp{B}=\xh$ .
\item[{\rm (ii)}] 
If $\supp{\nuh}=\xh$ and $B(\chi)$ is injective for $\nuh$-almost all
$\chi\in\xh$, then $\hh_K$  is dense in $\LduemuXK$ for any
probability measure $\mu$. \\
In the case $X$ is compact also the converse holds true.
\item[{\rm (iii)}] 
If $\kk=\complex$ and $B=1$, $\hh_K$  is dense in $\LduemuXK$ for any
probability measure $\mu$ if and only if $\supp \nuh= \hat{X}$.
\end{itemize}
\end{proposition}
\begin{proof}
Item~(i) follows from Proposition~\ref{necessita} and\eqn{insalata}.\\
Let now $\mu$ be a probability measure on $X$. Using (\ref{bochnervec}), we have
\begin{eqnarray}
\scal{L_\mu \varphi}{\varphi} & = & \iiint \chi(t-x)\scal{B(\chi) \varphi (t)}{\varphi (x)} \de\nuh (\chi) \de\mu (x) \de\mu (t) \notag \\
& = & \int_{\xh} \scal{B(\chi) \F{\varphi \mu} (\chi^{-1})}{\F{\varphi \mu}
  (\chi^{-1})} \de\nuh (\chi) . \label{elle mu} 
\end{eqnarray}
\begin{itemize}
\item[{\rm (ii)}] If $B(\chi)$ is injective for almost all
$\chi\in\xh$ and $\supp{\nuh}=\xh$, then, by the above equation,
positivity of $B(\chi)$ and
the injectivity of Fourier transform, $L_\mu \varphi \neq 0$ if $\varphi
\neq 0$ in $\LduemuXK$. Therefore, $\hh_K$ is dense in $\LduemuXK$ for
any probability measure $\mu$.\\
Suppose $X$ is compact, so that $\xh$ is discrete. If $\hh_K$  is
dense in $\LduemuXK$ for any probability measure $\mu$, ${\rm supp}\,\nuh =
\xh$ by item~(i). If $\chi_0\in\xh$ and $y\in\ker B(\chi_0)$, choose $\de \mu (x)
= \de x$ and $\varphi (x) = \overline{\chi_0 (x)} y$, so that
$\F{\varphi \mu} (\chi) = \delta_{\chi , \chi_0^{-1}} \, y$. We thus
have 
\[
\scal{L_\mu \varphi}{\varphi} = \scal{B(\chi_0) y}{y} \nuh ({\chi_0}) = 0.
\]
Since $L_\mu$ is injective, this implies $\varphi = 0$, {\em i.e.}~$y=0$.
\item[{\rm (iii)}] Since $B=1$,
  the `if' part is clear from item {\rm (ii)}. The converse follows by
  item~(i). 
\end{itemize}
\end{proof}
By inspecting the proofs of Propositions \ref{necessita} and
\ref{suff2}, one can easily replace $\LduemuXK$ with any $\LpmuXK$,
$1\leq p<\infty$, in the statements. The same holds for Corollary
\ref{sufffourier} below. 
\begin{remark}\rm If the translation invariant kernel $K$ is $\Co$,
  then Propositions \ref{necessita} and 
  \ref{suff2} characterize universality of $K$. 
\end{remark}
\begin{remark}\rm If $X= \real^d$, $\kk=\cuno$, and $\supp{\nuh}$ is
  a subset of $\xh = \real^d$ such that every entire function on
  $\complex^d$ vanishing on it is identically zero, then $K$ is
  $\kappa$-universal (see \cite[Proposition 14]{mixuzh06}). 
This follows by (\ref{elle mu}), taking into account that, for
compactly supported $\mu$, the Fourier transform of $\varphi \mu$ can
be extended to an entire function defined on $\cuno^d$.  \\ 
In particular, if $d=1$ a sufficient condition for
$\kappa$-universality is that $\supp{ \nuh}$ has an accumulation
point. 
\end{remark}
Based on the above remark, we give another example of
compact-universal kernel, which is not universal, see also
Example~\ref{controesempio}.
\begin{example}
\label{sync}
Let $K:\runo\times\runo\to\complex$ be the $\Co$-kernel 
\begin{eqnarray*}
K\lft x , t \rgt = \int_{-1}^{1}e^{2\pi i (t-x)p} \de p
= \frac{\sin{2\pi(t-x)}}{\pi (t-x)}, 
\end{eqnarray*}
with $\nuh$ the restriction of the Lebesgue measure to
$\left[ -1,1 \right]$.  Since the support of
$\nuh$ admits an accumulation point, $K$ is compact-universal by the
last remark. On the other hand since $\supp{\nuh}$ is not the
whole $\real$, $K$ is not universal by Proposition~$\ref{suff2}$.
\end{example}

We now exhibit a particular case in which Proposition \ref{suff2} applies.
\begin{corollary}\label{sufffourier}
Let $K$ be a translation invariant Mercer kernel such that
$\scal{K_0(\cdot)y}{y^\prime}\in L^1 (X,\de x)$ for all
$y,y^\prime\in\kk$. Let $B: \xh \frecc \BK$ be as in $(\ref{B se K e'
  L1})$. If $B(\chi)$ is injective for $\de \chi$-almost all $\chi$,
then the reproducing kernel Hilbert space $\hh_K$ is dense in
$\LduemuXK$ for any probability measure $\mu$. 
\end{corollary}
\begin{proof} 
Since the support of the Haar measure $\de\chi$ is $\xh$, the claim is
then a consequence of Proposition~\ref{suff2}. 
\end{proof}

\section{Examples of universal kernels}\label{secexample}

In this section we present various examples of universal kernels, some
of them has been already introduced in Section~\ref{tre}.

We start with the gaussian kernel, which is a well known example of
universal kernel. The first proof about universality is given
\cite{ste01} with a different technique and in \cite{mixuzh06} by
means of the Fourier transform. In both paper only
compact-universality is taken into account.
\begin{example}\label{gaussian}
Let $X$ be a closed subset of $\runo^d$, $\kk=\cuno$ and
\[\kappa(x,t)=e^{-\frac{\no{x-t}^2}{2\sigma^2}}\qquad x,t\in X,\]
where $\sigma>0$. Then $K$ is a $\Co$-universal kernel.
\end{example}
\begin{proof}
Assume first that $X=\runo^d$, regarded as abelian group, then
$\kappa$ is translation invariant kernel with $\kappa_0$ in
$\cccc_0(\runo^d)\cap L^1(\runo^d,dx)$. According to\eqn{B se K e' L1}
$$B(p)=\sqrt{(2\pi \sigma^2)^d}\,e^{-2\pi^2\sigma^2\no{p}^2}$$
where the dual group is identified with $\runo^d$ by means of
$\chi_p(x)=e^{i2\pi p\cdot x}$. 
Since $B(p)>0$ for all
$p\in\runo^d$, universality is a consequence of
Corollary~$\ref{sufffourier}$. \\ 
If $X$ is an arbitrary closed subset of $\runo^d$ it is enough to
apply Corollary$~\ref{HZu}$. 
\end{proof}
Next example is well known in functional analysis (see, for example,
\cite{bre83}). 
\begin{example}\label{sobolev}
Let $X=\runo$, $\kk=\cuno$ and let
\[\kappa(x,t)=e^{-\pi |x-t|}\,.\]
Then the kernel $\kappa$ is a $\Co$-universal kernel and
$\hh_\kappa=W^1\lft\real\rgt$, the 
Sobolev space of measurable complex functions $f$ on $\real$ with finite norm
$$
\no{f}^2_{W^1} = \int_{X} \left[ |f(x)|^2 + \left| f^\prime (x)
  \right|^2 \right] \de x , 
$$  
where $f^\prime$ is the weak derivative. 
\end{example}
\begin{proof}
The same reasoning as above, observing that
$B(p) =  \frac{2}{\pi + 4\pi p^2} > 0$ for all $p\in\runo$.
\end{proof}
Next example characterizes universal kernels of the form $K=\kappa B$
-- see Example~\ref{esempioB}.
\begin{example}\label{esempioBu}
Let $\kappa$ be a $\Co$-scalar reproducing kernel and $B$ a positive
operator. The kernel $K=\kappa B$ is universal if and only if
$\kappa$ is universal and $B$ is injective. 
\end{example}
\begin{proof}
We have to show that, given a 
probability measure $\mu$, $\hh_{\kappa
  B}$ is dense in $\LduemuXK$. The space $\hh_{\kappa
  B}$ is unitarily equivalent to $\hh_\kappa\otimes \ker{B}^{\perp}$ by
means of $W(\varphi\otimes y)(x)=\varphi(x)
B^{\frac 12}y$, see Example \ref{esempioB}. Hence, it is enough to
prove that $\hh_\kappa\otimes B^{\frac12}\kk$  
is dense in $L^2(X,\mu)\otimes \kk$. This is the case if
and only if $\hh_\kappa$ is dense in $L^2(X,\mu)$ and $B^{\frac12}$ has
dense range, and this last condition is equivalent to the fact that
$B$ is injective since $B$ is a positive operator.
\end{proof}
The same result holds replacing $\Co$-kernel with Mercer kernel and
universality with compact-universality.

\begin{example}
\label{univprod}
Let $\kappa:{X}\times {X}\to\cuno$ and $\kappa^\prime:{X^\prime}\times
{X^\prime}\to\cuno$ be two scalar $\Co$ reproducing kernels on ${X}$
and ${X^\prime}$, respectively. Let $I^\prime$ be the identity
operator on $\hh_{\kappa^\prime}$. 
\begin{itemize}
\item[{\rm (i)}] \,The $\hh_{\kappa^\prime}$-kernel $K=\kappa
I^\prime$ if universal if and only if $\kappa$ is universal.
\item[{\rm (ii)}]\,
Fixed a probability measure $\mu^\prime$ on ${X^\prime}$, the
$L^2({X^\prime},\mu^\prime)$-kernel 
$\widehat{K}=\kappa L_{\mu^\prime}$ is universal if and only if $\kappa$ is
universal and $\hh_{\kappa^\prime}$ is dense in $L^2({X^\prime},\mu^\prime)$.
\item[{\rm (iii)}]\,The scalar kernel $\kappa \times
  \kappa^\prime$ is universal if both $\kappa$ 
and $\kappa^\prime$ are universal.
\end{itemize}
\end{example}
\begin{proof}
Items {\rm (i)} and {\rm (ii)} follow immediately from Example
\ref{esempioBu} and Proposition \ref{densita-iniettivita}. Item {\rm
  (iii)} is a consequence of Proposition~\ref{product} and the density of 
$\Co(X)\otimes \Co(X^\prime)$ in $\Co(X\times X^\prime)$. 
\end{proof}

The following class of examples is considered in \cite{camipoyi08}.
\begin{example} Let $X$ be a locally compact second countable abelian
  group. Let $\lfg B^i \rgg_{i=1}^N$ be a finite set of
  positive operators on $\kk$ and $\lfg \kappa_0^i\rgg_{i=1}^N$ be a finite
  set of scalar functions of positive type in $\cccc_0(X)\cap
  L^1(X,\de x)$. The translation invariant kernel $K$
\[ K(x,t)=\sum_{i=1}^N \kappa_0^i(x-t) B^i\]
is universal provided that $\cap_i {\rm ker} B^i=\lfg 0\rgg$ and, for
each $i=1,\dots N$, 
there is an open dense subset $\hat{Z}^i\subset\hat{X}$ such that 
$\F{\kappa_0^i} > 0$ on $\hat{Z}^i$.
\end{example}
\begin{proof}
Clearly, $\scal{K_0(\cdot)y}{y^\prime}$ is in $L^1(X,\de
x)$. Moreover,  according to\eqn{B se K e' L1},
for all $y\in \kk$ and $\chi\in\hat X$ 
\begin{equation*}
B( \chi) y= \sum_{i=1}^N \F{\kappa_0^i}(\chi^{-1}) B^iy .
\end{equation*}
Each $\widehat{Z}^i$ is open and dense, hence $\widehat{Z}=\cap
\widehat{Z}^i$ is dense in 
$\hat{X}$. Let $\chi\in \hat{Z}$ and $y\in\kk$ such that $B(\chi)y=0$;
then $B^i y=0$ for all $i=1,\ldots, N$, since every $B^i$ is a
positive operator and $\F{\kappa_0^i} > 0$ on $\hat{Z}^i$, so that by
assumption $y =0$. Therefore, $K$ is universal by
Corollary~$\ref{sufffourier}$.   
\end{proof}
\appendix

\section{Vector valued measures}\label{sec. mis. vett.}

In this appendix we describe the dual of $\CoXK$. For $\kk=\complex$,
it is a well known result that ${\mathcal C}_0(X)^\ast$ can be
identified with the 
Banach space of complex measures on $X$. For arbitrary $\kk$, 
a similar result holds by considering the space of 
vector measures. If $X$ is compact, this result is due to
\cite{sin57} and we slightly extend it to $X$ being only locally
compact. The proof we give is  simpler than the original one
also for $X$ compact.  

Moreover, by using a version of Radon-Nikodym theorem for vector
valued measures, it is possible to describe the dual of $\CoXK$ in a
simpler way. Indeed, the following result holds.
\begin{theorem}\label{semplice}
Let $T\in \CoXK^\ast$. There exists a unique probability measure $\mu$ on $X$
and a unique function $h\in\LinfmuXK$ such that
\begin{equation}
  \label{forse}
   T(f)=\int_X \scal{f(x)}{h(x)} \de \mu (x)\qquad f\in\CoXK\
\end{equation}
with $\no{h(x)}=\no{T}$ for $\mu$-almost all $x\in X$. 
\end{theorem}
\begin{proof}
It follows combining Theorems \ref{Radon-Nikodym} and \ref{dualita'} below.
\end{proof}

Observe that, given $\mu$ and $h$ as in the statement of the
theorem, if we define $T$ by\eqn{forse}, then
$T\in\CoXK$. Hence\eqn{forse} completely characterizes the dual of
$\CoXK$ in terms of pairs $(\mu,h)$.

To prove the theorem, we recall some basic facts from the theory of vector
valued measures (see \cite{diuh77,lan93}). If $A\in\borx$,
we denote by $\Pi (A)$ the family of partitions of $A$ into finite or
denumerable disjoint Borel subsets.
\begin{definition}\label{caffe}
A {\em vector measure} on $X$ with values in $\kk$ is a
mapping $\M : \borx \frecc \kk$ such that 
\begin{itemize}
\item[{\rm (i)}]
$$
\sup_{ \{ A_i \} \in \Pi (X) } \sum_i \no{\M (A_i)}  < \infty ;
$$
\item[{\rm (ii)}] for all $A\in\borx$ and $\{ A_i \} \in \Pi (A)$
$$
\M(A) = \sum_i \M (A_i)
$$
where the sum converges absolutely by item {\rm (i)}.
\end{itemize}
\end{definition}

If $\M$ is a $\kk$-valued vector measure on $X$, for all $A\in\borx$ we define
$$
| \M | (A) = \sup_{ \{ A_i \} \in \Pi (A)} \sum_{i\in I} \no{\M (A_i)} .
$$
Then, $| \M |$ is a bounded positive measure on $X$, called
the {\em total variation} of $\M$. 

The integration of a function $f\in\LunoXMK$ with respect to $M$ is
defined as it follows. Let ${\rm St} (X ; \kk)$ be the space of
functions $f = \sum_{i = 1}^n 
1_{A_i} v_i$, with $A_i$ disjoint Borel sets and $v_i \in \kk$
($1_A$ is the characteristic function of the set $A$). For such
$f$'s, define 
\begin{equation}\label{int. vett.}
\int_{X} \scal{f(x)}{\de \M (x)} : = \sum_{i = 1}^n \scal{v_i}{\M (A_i)} .
\end{equation}
Since
\begin{equation*}
\left| \sum_{i = 1}^n \scal{v_i}{\M (A_i)} \right| \leq \sum_i
\no{v_i} \no{\M (A_i)} \leq \sum_i | \M | (A_i) \no{v_i} = \no{f}_{1}, 
\end{equation*}
the integral\eqn{int. vett.} extends to a bounded functional on
$\LunoXMK$, which is denoted again by $\int_{X} \scal{f(x)}{\de \M
  (x)}$. By Theorem 4.1 in \cite{lan93}, then there exists
$h\in\LinfXMK$ such that 
$$
\int \scal{f(x)}{\de \M (x)} = \int \scal{f(x)}{h(x)} \de | \M | (x) \quad \forall f\in\LunoXMK ,
$$
and $\no{h(x)} = 1$ for $| \M |$-almost all $x$.  These facts are collected in the following
theorem. 
\begin{theorem}[Radon-Nikodym]\label{Radon-Nikodym}
If $\M$ is a $\kk$-valued vector measure on $X$, there exists
a unique $| \M |$-measurable function $h : X\frecc \kk$ such that
$\no{h(x)} = 1$ for $| \M |$-almost all $x$ and 
$$
\int_{X} \scal{f(x)}{\de \M (x)} = \int_{X} \scal{f(x)}{h(x)} \de | \M
| (x) \quad \forall f\in\LunoXMK . 
$$
\end{theorem}
The function $h$ is called the {\em density} of $\M$ with respect to
$| \M |$.

We denote by $M(X;\kk)$ the space of $\kk$-valued vector measures on
$X$. The space $M(X;\kk)$ is a Banach space with respect to the norm 
$$
\no{\M} = |\M |(X) 
$$
(see \cite{diuh77}). If $\kk = \cuno$, we let $M(X) = M(X;\cuno)$.
The next duality theorem is shown in \cite{sin57} for $X$ compact --
see also \cite{diuh77}.
\begin{theorem}\label{dualita'}
If $\CoXK$ is endowed with the Banach space topology induced by the
uniform norm, then $\CoXK^\ast = M(X;\kk)$, the duality being given by 
\begin{equation*}
\scal{f}{\M} = \int_{X} \scal{f(x)}{\de \M (x)} \quad\forall f\in\CoXK
,\, \M\in M(X;\kk). 
\end{equation*}
\end{theorem}
\begin{proof}
By Theorem~\ref{Radon-Nikodym}, it is clear that, if $\M \in M(X;\kk)$, then
$$
T_{\M} (f) = \int_{X} \scal{f(x)}{\de \M (x)} = \int_{X}
\scal{f(x)}{h(x)} \de | \M | (x) 
$$
defines a bounded functional $T_{\M}$ on $\CoXK$. 

Clearly $\no{T_{\M}} \leq  \no{{\M}}$. To show that $\no{T_{\M}} =
\no{{\M}}$, fix by Lusin theorem a function $g\in\CoXK$ such that
$g(x) = h(x)$ for $x\in X\setminus Z$, $Z$ being a $|\M |$-measurable
set with $|\M | (Z) < \eps$, and $\no{g}_\infty \leq
\no{h}_{|M|,\infty} = 1$. For $\eps$ small enough, we then have  
$$
|\M | (X) - 2\eps < |\M | (X\setminus Z) - |\M | (Z) \leq \left|
  \int_{X} \scal{g(x)}{h(x)} \de | \M | (x) \right| \leq |\M | (X). 
$$
This shows that $\no{T_{\M}} = \no{{\M}}$.

Suppose now $T\in\CoXK^\ast$. For $v\in\kk$, let $i_v : \cccc_0 (X)
\frecc \CoXK$ be the bounded operator given by 
$$
[i_v (\varphi)](x) = \varphi (x) v.
$$
Since $T i_v \in \cccc_0 (X)^\ast$, by Riesz theorem there exists a
measure $\mu_v \in M(X)$ such that 
$$
T i_v (\varphi) = \int_{X} \varphi (x) \de \mu_v
(x)\qquad\text{and}\qquad \no{Ti_v}=\no{\mu_v}.
$$
For all $A\in\borx$, let $\M (A)$ be the vector in $\kk$ such that
$$
\scal{v}{\M (A)} = \mu_v (A)
$$
($\M (A)$ is well defined, since $|\mu_v (A) | \leq \no{\mu_v} = \no{T i_v} \leq \no{T} \no{v}$). 

We now show that, if $A\in\borx$ and $\{A_i\} \in\Pi(A)$, then
$$
\sum\nolimits_i \no{\M (A_i)} \leq \no{T},
$$
so that item~(i) of Definition~\ref{caffe} holds. It is enough to prove it for
all finite partitions $\{A_i\}_{i=1\ldots n}$. Let $v_i = \M (A_i) /
\no{\M (A_i)}$ (we set $v_i = 0$ whenever $\M (A_i) = 0$). We have 
\begin{equation*}
\sum\nolimits_i \no{\M (A_i)} = \sum\nolimits_i \scal{v_i}{\M (A_i)} = \sum\nolimits_i \mu_{v_i} (A_i).
\end{equation*}
Set $\nu = \sum_i | \mu_{v_i} |$, which is $\nu$ a bounded positive
measure, and every $\mu_{v_i}$ has density with respect to $\nu$. For
all $i = 1\ldots n$, fix a sequence $\{ \varphi^{(i)}_j \}_{j\in\nat}$
in $\cccc_c (X)$ such that $\lim_j \varphi^{(i)}_j (x) = 1_{A_i} (x)$
for $\nu$-almost all $x$. Define 
$$
\psi_j (x) = \left[ 1 \vee \sum_{k=1}^n \left| \varphi^{(k)}_j (x)
  \right| \right]^{-1}  \sum_{i=1}^n\varphi^{(i)}_j (x) v_i . 
$$
Then, $\psi_j \in \cccc_c (X;\kk)$, and $\no{\psi_j (x)} \leq 1$ for all $x$. Moreover,
$$
\left| \left[ 1 \vee \sum_{k=1}^n \left| \varphi^{(k)}_j (x) \right|
  \right]^{-1}  \varphi^{(i)}_j (x) \right| \leq 1 \quad
\forall x,\  i
$$
and
$$
\lim_j \left[ 1 \vee \sum_{k=1}^n \left| \varphi^{(k)}_j (x) \right|
\right]^{-1}
\varphi^{(i)}_j (x) = 1_{A_i} (x) \quad \textrm{for $\nu$-almost all $x$} .
$$
Therefore
\begin{eqnarray*}
&& \left| \sum\nolimits_i \no{\M (A_i)} - T\psi_j \right| \\
&& \qquad = \left| \sum\nolimits_i \left\{ \mu_{v_i}(A_i) - T i_{v_i} 
\left( \left[  1 \vee \sum_k \left| \varphi^{(k)}_j \right|
\right]^{-1} \varphi^{(i)}_j \right) \right\} \right| \\ 
&& \qquad \leq \sum\nolimits_i \left| \int_{X} \left\{1_{A_i} (x) - 
\left[ 1 \vee \sum_k \left| \varphi^{(k)}_j (x) \right| \right]^{-1}
\varphi^{(i)}_j (x) \right\} \de \mu_{v_i} (x) \right| \\ 
&& \qquad \stackrel{j\to\infty}{\frecc} 0
\end{eqnarray*}
by dominated convergence theorem.
On the other hand, $| T \psi_j | \leq \no{T} \no{\psi_j}_\infty \leq \no{T}$. It follows that $\sum_{i=1}^n \no{\M (A_i)} \leq \no{T}$, as claimed.

We now show that
$$
\M (A) = \sum_i \M (A_i)
$$
(absolutely) for all $A\in\borx$ and $\{ A_i \} \in \Pi (A)$. We have
just proved that the right hand side is absolutely convergent, and the
equality follows by 
$$
\scal{v}{\sum_i \M (A_i)} = \sum_i \mu_v (A_i) = \mu_v (A) = \scal{v}{\M (A)} \quad \forall v\in\kk .
$$

Therefore, $\M$ is a $\kk$-valued measure. It remains to show that $T
= T_{\M}$. Let $h$ and $|\M |$ be associated to $\M$ as in
Radon-Nikodym theorem. Then, for any Borel set $A\subset X$, we have 
$\mu_v (A) = \int_{A} \scal{v}{h(x)} \de |\M | (x)$, from which it follows that $\mu_v$ has density $\scal{v}{h(x)}$ with respect to $|\M |$. For $\varphi\in\cccc_c (X)$ and $v\in\kk$, we thus have
$$
T(\varphi v) = \int_{X} \varphi (x) \de \mu_v (x) = \int_{X} \scal{\varphi (x) v}{h(x)} \de |\M | (x) = T_{\M} (\varphi v) .
$$
Then, $T = T_{\M}$ by density of $\cccc_c (X) \otimes \kk$ in $\CoXK$.
\end{proof}

\begin{ack}
  This work has been partially supported by the FIRB project
  RBIN04PARL and by the the EU Integrated Project Health-e-Child
  IST-2004-027749.
\end{ack}

\end{document}